\documentclass[12pt]{amsart} 
\usepackage{latexsym, amssymb, amsfonts, amsmath, amsthm, graphics, graphicx, subfigure, bbm, stmaryrd, xcolor,url}

\usepackage{graphicx}
\usepackage{verbatim}
\usepackage{appendix}
\usepackage{float}

\newtheorem{theorem}{Theorem}[section]
\newtheorem{lemma}[theorem]{Lemma}

\theoremstyle{definition}

\DeclareMathOperator{\Cov}{COV}

\DeclareMathOperator{\Var}{Var}

\numberwithin{equation}{section}

\numberwithin{equation}{section}

\theoremstyle{remark}

\numberwithin{equation}{section}

\DeclareMathOperator*{\argmin}{arg\,min}

\begin{document}

\bibliographystyle{plain}


\title{An Analytical Formula for Spectrum Reconstruction}

\author{Zhibo Dai$^{*}$}
\author{Heinrich Matzinger}
\author{Ionel Popescu}
\address{School of Mathematics, Georgia Institute of Technology, Altanta, GA 30313}

\email{zdai37@gatech.edu}
\email{matzi@math.gatech.edu}
\email{ipopescu@math.gatech.edu}

\renewcommand{\shortauthors}{Z.Dai, H.Matzinger, I.Popescu}


\date{May.8 2019}


\begin{abstract}
{We study the spectrum reconstruction technique. As is known to all, eigenvalues play an important role in many research fields and are foundation to many practical techniques such like PCA(Principal Component Analysis). We believe that related algorithms should perform better with more accurate spectrum estimation. There was an approximation formula proposed by\cite{amsalu2018recovery}, however, they didn't give any proof. In our research, we show why the formula works. And when both number of features and dimension of space go to $\infty$, we find the order of error for the approximation formula, which is related to a constant $c$-the ratio of dimension of space and number of features.}
\end{abstract}

\maketitle

\section{Introduction}
\label{intro}
Our research is about a simple analytical formula for the difference between the sample covariance and ground truth covariance spectrum of large multivariate normal data. We let both of the sample size and the dimension, in which the
data lives, go to infinity at the same time. We show why a simple analytical approximation formula for the difference between sample spectrum and ground truth spectrum \ref{useful_approx} and similarly \ref{blue} holds in certain cases. These formulas have already been introduced in \cite{amsalu2018recovery}, but are without a justification of why they should hold. In section \ref{section_larger_order}, we show that the approximation \ref{blue} holds when a given condition \ref{condition_sample_size} on the size of the sample holds. This condition is for eigenvalues, which are of somewhat larger order at least
$O(n^{0.5})$ as we argue. Note that this result is not asymptotic and we use the result of Lounici and Koltchinskii \cite{koltchinskii2017normal} allowing to bound the error matrix in covariance estimation.

In Section \ref{section_large_constant}, we consider the situation where the sample size is a large constant times
the space dimension. We show that by taking the constant big enough, we get approximation \ref{useful_approx} to hold as good as we want (relative error as small as we want) on the inside.

Let us first give the background of the problem:
Assume  that $Z$ is a $n$ by $p$ data matrix. Assume, for example, that each column of $Z$ is a point in a machine learning problem. Note that the product matrix  $Z^t Z$ contains all the  inner products between columns. Here $Z^t$ represents the transpose of $Z$. It is then easy to see that from $Z^t Z$ we can find all the relative positions of the column vectors with respect to each other. We view them as vectors of $\mathbb{R}^n$. So, any machine learning algorithm, of which output depends only on the relative position of points to each other, would only need $Z^t Z$ as input rather than $Z$ given the input points for the algorithm are the columns of $Z$.

We consider a random matrix $Z$ with i.i.d. columns distributed each like a random vector $\vec{Z}=(Z_1,Z_2,\ldots,Z_p)$, which we assume to have zero expectation due to standardizing the data  done before using most machine learning algorithms. Then, the {\it sample covariance matrix } is defined as:
\begin{equation}
\label{sample_covariance}
\hat{\Cov}[\vec{Z}]=\frac{Z^t \cdot Z}{n}
\end{equation}
and is an unbiased estimate of the covariance $COV[\vec{Z}]$. The covariance matrix
\begin{equation}
\label{covariance}
\Cov[\vec{Z}]=(E[Z_i\cdot Z_j])_{ij}
\end{equation}
will be  called the {\it ground truth covariance}.
Now, the sample covariance as estimate of the ground truth covariance is very bad as long as $n<p$ since then \ref{sample_covariance} is defective, which means it has some eigenvalues approximating $0$ assuming $COV[\vec{Z}]$ has no zero eigenvalues. This is the problem that affects many machine learning algorithms and it is also called the curse of dimensionality. In traditional statistics, one assumes $p$ fixed while $n$ goes to infinity. In modern high dimensional statistics, one lets $n=c\cdot p$, where $c$ is a constant not depending on $p$.
This implies that both $n$ and $p$ go to infinity at the same time, which is the situation we consider in our research. In the case that $c$ is not large enough, as already mentioned the sample covariance is a very bad estimate of the true covariance since its small eigenvalues will be much smaller than the corresponding eigenvalues of ground truth covariance.

\begin{figure} \centering
\subfigure[]
  {\label{comparison_c_0.5}
  \includegraphics[width=0.47\textwidth]{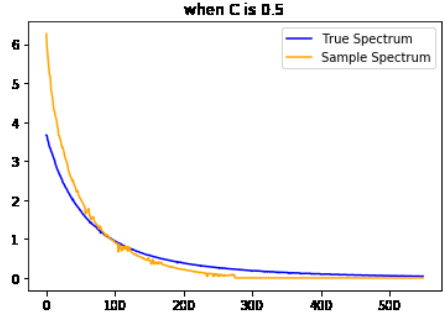}
  }
  \caption{Covariance Spectrum vs Ground Truth Spectrum when n = p/2}
\end{figure}

\begin{figure} \centering
\subfigure[]
  {\label{comparison_c_2}
  \includegraphics[width=0.47\textwidth]{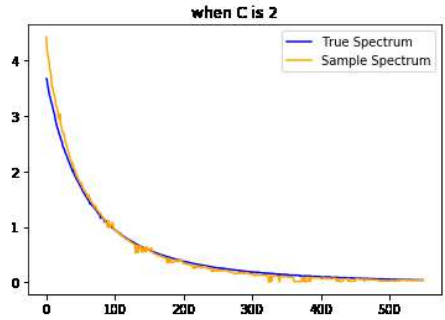}
  }
  \caption{Covariance Spectrum vs Ground Truth Spectrum when n = 2*p}
\end{figure}

Now we will assume that we are dealing with multivariate normal vector $\vec{Z}$. We denote by $\sigma^2_j$, the $j$-th eigenvalue of the covariance $\Cov[\vec{Z}]$ in descending order. So the spectrum of the covariance matrix is
\begin{equation}
\label{spectrum}
\sigma_1^2\geq \sigma_2^2\geq\ldots\geq\sigma_p^2
\end{equation}
and the corresponding eigenvalues of the sample covariance will  be denoted by $\hat{\sigma}^2_j$
and hence
\begin{equation}
\label{sample_spectrum}
\hat{\sigma}_1^2\geq \hat{\sigma}_2^2\geq\ldots\geq\hat{\sigma}_p^2.
\end{equation}
The topic of this chapter 1 is reconstruction of \ref{spectrum} when given only \ref{sample_spectrum}
in the context of normal data. Note that when both $n$ and $p$ go to infinity, the fluctuation of \ref{sample_spectrum} is of smaller order than the values themselves. So for practical purposes we can consider the spectrum \ref{sample_spectrum} to be non-random. In traditional statistics, where $p$ is fixed as $n$ goes to infinity the opposite is true. The unit-eigenvectors of the covariance matrix \ref{covariance}, are called {\it Principal Components}. When we represent the vector $\vec{Z}$ in  the coordinate system of the Principal Components
the new coordinates are uncorrelated. For a normal vector this implies independence. So, let
$$\vec{X}=(X_1,X_2,\ldots,X_p)$$
denote the vector $\vec{Z}$ expression in the Principal Components of $\Cov[\vec{Z}]$.
So the random vector $\vec{X}$ is a normal vector with independent normal components where
$$VAR[X_i]=\sigma^2_i$$
for $i= 1,2,\ldots,p$. So, we will also express the data matrix in the coordinate system of the PCA, which means that each row of $Z$ is going to be expressed in the basis of the PCA. Hence, we will not work with the data matrix $Z$, but instead with a data matrix $X$, where each row is an independent copy of $\vec{X}$. This means that
$X$ is a $n\times p$ matrix, with i.i.d rows where columns are also independent. In the $j$-th column we have normal entries with $0$ expectation and variance $\sigma_j^2$. So, from here on the sample covariance is defined by
$$\hat{\Cov}[\vec{X}]:=\frac{X^t\cdot X}{n}$$
and the ground truth covariance is the diagonal matrix
$$\Cov[\vec{X}]=\left(
\begin{array}{ccccccccc}
\sigma_1^2&0&0&\ldots&0\\
0&\sigma^2_2&0&\ldots&0\\
0&0&\sigma^2_3&\ldots&0\\
\ldots\\
0&0&0&\ldots&\sigma^2_p
\end{array}\right)
$$
We will also designate the ground truth covariance by $\Sigma_p$
and hence
$$\Sigma_p:=\Cov[\vec{X}]$$

Now let $\Sigma_p^{1/2}$ be the  square root of $\Sigma_p$::
$$\Sigma^{1/2}_p=\sqrt{\Cov[\vec{X}]}=\left(
\begin{array}{ccccccccc}
\sigma_1&0&0&\ldots&0\\
0&\sigma_2&0&\ldots&0\\
0&0&\sigma_3&\ldots&0\\
\ldots\\
0&0&0&\ldots&\sigma_p
\end{array}\right)$$
Then the data matrix $X$ has same distribution as $N\cdot \Sigma^{1/2}_p$,
where $N$ is a $n$ times $p$ matrix of i.i.d. standard normal entries.
Hence, we can write the sample covariance as
\begin{equation}
\label{sample_cov_X}
\hat{\Cov}[\vec{X}]=\frac{X^t\cdot X}{n}=\frac{\Sigma^{1/2}_p N^t\cdot N \Sigma^{1/2}_p}{n}
\end{equation}
Now, we use the property that the product of square matrices $A\cdot B$ has same spectrum as $B\cdot A$. We let $A$ be the matrix$\Sigma^{1/2}_p$, and $B$ be $N^t\cdot N \Sigma^{1/2}_p$. So, after applying the rule that $AB$ and $BA$ have same spectrum we find that the spectrum of
\begin{equation}
\label{N^tNSigma}
BA=\frac{(N^t\cdot N)}{n}\cdot \Sigma_p
\end{equation}
is identical with the spectrum of the sample covariance \ref{sample_cov_X}, which is $AB$.

Now free probability theory tells us that the spectrum of the product of two independent symmetric random square matrices converges as long as each matrix's spectrum converges. Actually in order to have the convergence, we need random matrices to be unitary invariant, which is indeed the case here as $N^t N$ satisfies the condition. For this we let the dimension of the matrices go to infinity. During that convergence process, their spectrum converges to a "finite" limiting distribution.  The limit of the product's spectrum is called the {\it Free Product} of the limiting distributions of each spectrum taken separately. We can apply this to the product on the right side of \ref{N^tNSigma} as long as the spectrum of $\Sigma_p$ converges, when $p$ goes to infinity. In that case, the product must converge to a free product. That is the free product of the limiting distribution for $\Sigma_p$, with that  of $\frac{(N^t\cdot N)}{n}$.
In 1967, Vladimir Mar{\v c}enko and Leonid Pastur \cite{marvcenko1967distribution} successfully constructed the limiting law of $\frac{(N^t\cdot N)}{n}$, which is now named after the authors, namely Marchenko–Pastur distribution. This is the case when both $n$ and $p$ go to infinity at the same time. Therefore, $\frac{n}{p}$ converges to a non-zero fixed limit, which we denote by $c$. The limiting law depends on $c$.

So, from our explanation of free probability, in the  case of  $\Cov[{X}]$ 's spectrum admitting a limiting law $F^{\Sigma}$, the sample spectrum is a free product of Mar{\v c}enko-Pastur law and $F^{\Sigma}$. By computing the $S$-transform explicitly, one can obtain a formula of the limiting law of the sample spectrum. See Bai and Yin \cite{bai1988convergence}, Yin, Bai and Krishnaiah \cite{yin1983limiting}, Silverstein \cite{silverstein1995strong}, and many others. The main result is summarized as follows:

\begin{theorem}
	Given the following conditions,
	\begin{enumerate}
		\item{} 	Suppose entries of $N_p=(N_{i,j})_{n\times p}$ are i.i.d. real random variables for all $p$.
		
		\item{} $E[N_{1,1}]=0$, $E[|N_{1,1}|^2]=1$.
		
		\item{} Let $n/p \to c >0$ as $p\to \infty$.
		
		\item{} Let $\Sigma_p$ $(p\times p)$ be non-negative definite symmetric random matrix with spectrum distribution $F^{\Sigma_p}$ (If $\{\lambda_i\}_{1\le i \le p}$ are the eigenvalues of $\Sigma_p$, then $F^{\Sigma_p}=\sum_{1}^{p} \frac{1}{p} \delta_{\lambda_i}(x)$)  such that $F^{\Sigma_p}$ almost surely converges weakly to $F^{\Sigma}$ on $[0, \infty)$.
		
		\item{} $N_p$ and ${\Sigma}_p$ are independent.
	\end{enumerate}
	then the spectrum distribution of $W_p= \frac{1}{n}{\Sigma}_p^{1/2}N_p^T N_p {\Sigma}_p^{1/2}$, denoted as $F^{W_p}$  almost surely converges weakly to $F^W$. $F^W$ is the unique probability measure whose Stieltjes transform $m(z)= \int \frac{d F^W(x)}{x-z}$, $z\in \mathbb{C}^+$ satisfies the equation
	\begin{equation} \label{eqn:sample spectrum}
	-\frac{1}{m}=z- c\int \frac{t }{1+tm} d F^{\Sigma}(t) \quad \forall z \in \mathbb{C}^+
	\end{equation}
	
\end{theorem}

So, recall the topic of this chapter is reconstructing the ground truth spectrum given only the spectrum of the sample covariance. Currently most methods for this problem are "free-probability based". That is they attempt to  solve the equation \ref{eqn:sample spectrum} to get an estimator of the true spectrum $F^{\Sigma}$. Take the sample covariance spectrum as if it would be the limit, which is taking $F^{W_p}$ for the distribution $F^W$ in order to
get the Stieltjes transform $m(z)$. So, instead of $m(z)$, we use $m_p(z)$ in \ref{eqn:sample spectrum}.
Then solve, which is finding $F^\Sigma$ solving \ref{eqn:sample spectrum} and  pretending it is $F^{\Sigma_p}$. In such an approach one hopes that there is only a little difference between the observed spectral distribution for a given $p$ and the limiting distributions. Such an approach based on free probability theory was pioneered by EI Karoui \cite{el2008spectrum}, and then Bai etc. \cite{bai2010estimation}, and recently by Ledoit and Wolf \cite{ledoit2012nonlinear}and \cite{ledoit2015spectrum}. It's not surprising that as dimensions grow, consistency is achieved by the free probability approach. But a disadvantage is that the recovered spectrum is still far from the true spectrum for small or moderate size of $p$ since the method operates as if the data given would already be in the "free-probability limit". Another problem with the free probability approach is that the spectrum of the ground truth covariance needs to converge for free probability to be applicable. However, in real data, there is always different order eigenvalues: some have order $O(1)$ and some have order $O(p)$.
 
In\cite{amsalu2018recovery}, Matzinger etc. proposes 2 methods to reconstruct population spectrum based on sample spectrum. The first one is a simple algebraic formula to reconstruct population eigenvalues given sample eigenvalues. It is our formula \ref{useful_approx}. Unless the structure of population spectrum is too flat, this estimation performs well. Note that our formula can be interpreted as: the relative error between sample spectrum and ground truth spectrum is approximately the Stieltjes transform of sample spectrum. Indeed, our approximation \ref{useful_approx} can be rewritten:
 \begin{equation}
 \label{stiljes_matzi}
 c\cdot \frac{\hat{\sigma}_i^2-\sigma_i^2}{\sigma^2_i}\approx
 -\frac{1}{p}\sum_{j\neq i}
 \frac{\hat{\sigma}_{j}^2}
 {\hat{\sigma}_{j}^2-\hat{\sigma}_i^2}=-1-\frac{\hat{\sigma}^2_i}{p}\sum_{j\neq i}
 \frac{1}
 {\hat{\sigma}_{j}^2-\hat{\sigma}_i^2}=-1 -\hat{\sigma}^2_i \int  \frac{1}{x-\hat{\sigma}^2_i}dF^{W_p}(x),
 \end{equation}
 where for the integral over the spectral measure $F^{W_p}$ we make the convention that we leave out the atom at $\hat{\sigma}^2_i$ since otherwise we would have $0$ in the denominator of the summation. (Also, as usual we have $n=c\cdot p$). Note that on the very right side of \ref{stiljes_matzi}, we have the Stieltjes transform of the empirical distribution of the sample covariance spectrum. The second method proposed by Matzinger etc.\cite{amsalu2018recovery} is a fixed point method. The second approach is more computation-expensive but achieves a more accurate estimate, and we won't treat it here. Our research focus on the first approach proposed by  Matzinger etc all \cite{amsalu2018recovery}. We do a deeper analysis on the error term and theoretically show it is negligible under certain condition. We also present a similar formula \ref{blue} but where sample covariance and covariance spectrum are inverted.
 
In most situations, researchers use the spectrum extracted from a sample matrix, especially the sample covariance matrix, which brings in some error due to sample estimation bias. Therefore, estimating the eigenvalues of a population covariance matrix from a sample covariance matrix is of fundamental importance. The population spectrum will provide us more accurate essential information about the structure of the data problem\cite{burda2004signal}.

\section{Related Work}
Both of eigenvalues and eigenvectors have significant influence in mathematics and real life. Theoretically, they can be applied in linear algebra, differential operators and dynamic equations such as matrix diagonalization, eigen decomposition, eigenvector-eigenvalue identity and solving differential equations. Apart from mathematics, researchers also utilize the properties of eigenvalue and eigenvector in Schrödinger equation\cite{cooley1961improved}, geology\cite{graham2000graphical} and vibration analysis. The widely known application would be Principal Component Analysis(PCA)\cite{jolliffe1986principal}, which is used in dimension reduction\cite{kambhatla1997dimension}, feature selection\cite{malhi2004pca,song2010feature,guo2002feature}, K-means clustering\cite{ding2004k} and general text classification problems\cite{gomez2012pca,zahedi2013improving,uuguz2011two,nedungadi2014high}.

In most situations, researchers use the spectrum extracted from a sample matrix, especially the sample covariance matrix, which brings in some error due to sample estimation bias. Therefore, estimating the eigenvalues of a population covariance matrix from a sample covariance matrix is of fundamental importance. The population spectrum will provide us more accurate essential information about the structure of the data. There are a family of researches on spectrum reconstruction algorithms, which attempt to discretize and adapt the free probability infinite dimensional recovery. The idea was first introduced by \cite{el2008spectrum} based on the Marčenko–Pastur equation. In\cite{yin1983limiting}, authors show that the spectral distribution of a central multivariate matrix converges to a a limit distribution in probability. Then in\cite{bai1988convergence} the convergence of the spectral distribution of the sample covariance matrix to the semicircle law was proved given the assumption that $X_{p} = [X_{ij}]_{p*n}$ has iid entries and $E(X_{11}^{4})<\infty, var(X_{11})=1$. A similar result on strong convergence of the empirical distribution of eigenvalues was proved in 1995 by \cite{silverstein1995strong}. Recently, in\cite{ledoit2012nonlinear}\cite{ledoit2015spectrum} authors propose a novel estimate of the population eigenvalues which is consistent under large-dimensional asymptotics regardless of whether or not they are clustered, and that also performs well in finite sample. They find the estimate by solving the following optimization problem:
$$\hat(\tau_{n}) = \argmin_{t\in [0,\infty)^{p}}\frac{1}{p}\sum_{i=1}^{p}[q_{n,p}^{i}(t)-\lambda_{n,i}]^{2}$$
where $Q_{n,p}(t) = (q_{n,p}^{1}(t),\dots,q_{n,p}^{p}(t))^t$ is the nonrandom QuEST function. And this convergence is almost surely convergence. In\cite{li2013estimation} researchers show another new method founded on a meaningful generalization of the seminal Marcenko-Pastur equation, originally defined in the complex plan, to the real line.

Shrinkage is also one of methods to reconstruct population spectrum. The idea was pioneered by Stein\cite{stein1975estimation}. See also Bickel\cite{bickel2008regularized} and Donoho\cite{gavish2017optimal}. Another type of approach is based on the moments of the spectral distributions\cite{kong2017spectrum}, which shows a theoretically optimal and computationally efficient algorithm for recovering the moments of the population eigenvalues. Finally, there are also Physicists Burda, Gorlich and Jarosz, working on this problem\cite{burda2004signal}. 

\section{Case of Larger Order Eigenvalues}
\label{section_larger_order}
In this subsection we show the approximation formula \ref{blue} to hold when the condition \ref{condition_sample_size} is satisfied. So we first write down a three dimensional vector but the formula will still be useful in high dimension case. Now we have a sequence of i.i.d. vectors with $0$ expectation:
$$\vec{X},\vec{X}_1,\vec{X}_2,\ldots,\vec{X}_n$$
where $\vec{X}_i=(X_i,Y_i,Z_i)$ and
$\vec{X}=(X,Y,Z)$

We will assume that
\begin{eqnarray*}
E[\vec{X}]\\
&=&E[(X,Y,Z)]\\
&=&E[\vec{X}_i]\\
&=&(E[X_i],E[Y_i],E[Z_i])\\
&=&(0,0,0)
\end{eqnarray*}
We will explain later why in many applications, this assumption is realistic. We assume that $X_i$, $Y_i$ and $Z_i$ are independent of each other. Hence, the covariance matrix is given by
\begin{align*}
\Cov[\vec{X}]=
\left[
\begin{array}{ccccccccc}
\sigma_X^2&0&0\\
0&\sigma_Y^2&0\\
0&0&\sigma_Z^2
\end{array}
\right]=
\left[
\begin{array}{ccccccccc}
E[X^2]&E[XY]&E[XZ]\\
E[YX]&E[Y^2]&E[YZ]\\
E[ZX]&E[ZY]&E[Z^2]
\end{array}
\right]
\end{align*}

Now recall the Central Limit Theorem: Assume we have variables $W_1,W_2,\ldots$, which are i.i.d, then we have for $n$ large enough, the properly re-scaled sum is approximately standard normal:
$$\frac{W_1+W_2+\ldots+W_n-nE[W_1]}{\sqrt{n}\sigma}\approx \mathcal{N}(0,1)$$
the goal is to figure out how precise our estimates for the eigenvalues and eigenvectors are. Since the expectation is $0$, in our estimate of the covariance matrix we can leave the part which estimates the expectation out. Then we use the following estimate for the covariance matrix:
$$\hat{\Cov}[\vec{X}]=
\left[
\begin{array}{ccccccccc}
\frac{X_1^2+\ldots+X_n^2}{n}&\frac{X_1Y_1+\ldots+X_nY_n}{n}&
\frac{X_1Z_1+\ldots+X_nZ_n}{n}\\
\frac{Y_1X_1+\ldots+Y_nX_n}{n}&\frac{Y_1^2+\ldots+Y_n^2}{n}&
\frac{Y_1Z_1+\ldots+Y_nZ_n}{n}\\
\frac{Z_1X_1+\ldots+Z_nX_n}{n}&\frac{Z_1Y_1+\ldots+Z_nY_n}{n}&
\frac{Z_1^2+\ldots+Z_n^2}{n}\\
\end{array}
\right]$$

We can now apply the Central Limit Theorem to all entries
of the estimated covariance matrix above. For example let's take $W_i$ to be $W_i=X_i Y_i$. Then
\begin{align}\label{lalube}
\frac{X_1Y_1+\ldots+X_nY_n}{n}-E[X_1Y_1]
=& \frac{1}{\sqrt{n}}\frac{W_1+W_2+\ldots+W_n-E[W_1]}{\sqrt{n}}\nonumber\\
&\approx \sigma_{W_1}\frac{\mathcal{N}(0,1)}{\sqrt{n}}\nonumber \\
&=\frac{\sigma_{X_1}\sigma_{Y_1}}{\sqrt{n}}\mathcal{N}(0,1)
\end{align}
So take the difference $E$ between the estimated covariance matrix
and the real one, being called the covariance estimation matrix:
\begin{align*}
E&=\hat{\Cov}[\vec{X}]-\Cov[\vec{X}]\\
&=\left[
\begin{array}{ccccccccc}
\frac{X_1^2+\ldots+X_n^2}{n}\;\;\;\;\;-E[X^2]&\frac{X_1Y_1+\ldots+X_nY_n}{n}-E[XY]&
\frac{X_1Z_1+\ldots+X_nZ_n}{n}-E[XZ]\\
\frac{Y_1X_1+\ldots+Y_nX_n}{n}-E[YX]&\frac{Y_1^2+\ldots+Y_n^2}{n}\;\;\;\;\;\;-E[Y^2]&
\frac{Y_1Z_1+\ldots+Y_nZ_n}{n}\;\;-E[YZ]\\
\frac{Z_1X_1+\ldots+Z_nX_n}{n}-E[ZX]&\frac{Z_1Y_1+\ldots+Z_nY_n}{n}\;\;-
E[ZY]
&
\frac{Z_1^2+\ldots+Z_n^2}{n}\;\;\;\;\;\;\;\;-E[Z^2]\\
\end{array}
\right]
\end{align*}

With the Central Limit Theorem applied to each of the entries of the last matrix above in the same way as in \ref{lalube}.Now, let $N_{ij}$ be the re-scaled $i,j$-th entry of our covariance estimation error matrix. Hence, $$N_{12}=\frac{\sqrt{n}E_{12}}{\sigma_X\sigma_Y},N_{13}=\frac{\sqrt{n}E_{13}}{\sigma_X\sigma_Z},N_{23}=\frac{\sqrt{n}E_{23}}{\sigma_Y\sigma_Z}$$
and
$$N_{11}=\frac{\sqrt{n}E_{11}}{\sigma^{2}_{X}},N_{22}=\frac{\sqrt{n}E_{22}}{\sigma^{2}_{Y}},N_{33}=\frac{\sqrt{n}E_{33}}{\sigma^{2}_{Z}},$$
whilst $N_{ij}=N_{ji}$.  By definition, the term $N_{ij}$ has expectation $0$ and standard deviation $1$.

Clearly as $n$ goes to $\infty$, the $N_{ij}$ is asymptotically standard normal.
With this notation:
\begin{equation}
\label{Covapprox}\hat{\Cov}[\vec{X}]-\Cov[\vec{X}]=
\frac{1}{\sqrt{n}}
\left[
\begin{array}{ccccccccc}
\sigma^{2}_{X}N_{11}&\sigma_X\sigma_YN_{12}&
\sigma_X\sigma_ZN_{13}\\
\sigma_Y\sigma_XN_{21}&\sigma^{2}_{Y}N_{22}&
\sigma_{Y}\sigma_{Z}N_{23}\\
\sigma_Z\sigma_XN_{31}&\sigma_Z\sigma_YN_{32}&
\sigma^{2}_{Z}N_{33}
\end{array}
\right]
\end{equation}
Also, note that the terms $N_{11},N_{22},N_{33},N_{12},N_{13},N_{23}$ are all pairwise uncorrelated. For example:
\begin{align*}
\Cov(XY,XZ)\\
&=E[XYXZ]-E[XY]\cdot E[XZ] \\
&=E[X^2]E[Y]E[Z]-E[X]E[Y]E[X]E[Z] \\
&=0
\end{align*}
Hence,
\begin{align*}
&\Cov\left(\frac{X_1Y_1+\ldots+X_nY_n}{n},\frac{X_1Z_1+\ldots+X_nZ_n}{n}
\right)\\
&=\frac{1}{n^2}\sum_{i,j}\Cov(X_iY_i,X_jZ_j)=
\frac{1}{n^2}\sum_{i}\Cov(X_iY_i,X_iZ_i)\\
&=0
\end{align*}

Next we are going to establish the formula for the estimated eigenvalue and eigenvectors of the covariance matrix. Again, the estimated eigenvalues and eigenvectors are simply the eigenvectors and eigenvalues of the estimated covariance matrix. We assume $\sigma_X$, $\sigma_Y$ and $\sigma_z$ all have different values. Let $A$ denote the covariance matrix, $E$ again the error-matrix, which is the difference between the estimated and the true covariance matrix. Let $\vec{\mu}=(1,0,0)^T$ be the first eigenvector of $A=\Cov[\vec{X}]$. Let $\lambda=\sigma_X^2$ denote the first eigenvalue of the covariance matrix $A$ and let $\lambda+\Delta\lambda$ denote the first eigenvalue of the estimated covariance matrix.

So the estimated covariance matrix is $A+E$, hence the true covariance matrix plus a ``perturbation'' $E$. Let $\vec{v}=\vec{\mu}+\Delta\vec{\mu}$ be the first eigenvector for the estimated covariance matrix and assume that $\Delta\vec{\mu}$ is orthogonal to $\mu$. Hence $\Delta\vec{\mu}=(0,\Delta\mu_Y,\Delta\mu_Z)^T$. With these notations, we have:
\begin{equation}
\label{A+E}
(A+E)(\vec{\mu}+\Delta\vec{\mu})
=(\lambda+\Delta\lambda)(\vec{\mu}+\Delta\vec{\mu}).
\end{equation}
Also, since $\vec{\mu}$ is an eigenvector of $A$,
we have:
\begin{equation}
\label{Amu}
A\vec{\mu}=\lambda\vec{\mu}
\end{equation}
Subtracting equation \ref{A+E} from \ref{Amu},
we find:
\begin{equation}
\label{noise}
(A-I\lambda)\Delta\vec{\mu}=-E\vec{\mu}+\Delta\lambda\vec{\mu}+
-E\Delta\vec{\mu}+\Delta\lambda\Delta\vec{\mu}.
\end{equation}

we find the following exact equation:
\begin{align*}
\label{zaza3}
&\left[
\begin{array}{ccccccccc}
0&0&0\\
0&\sigma_Y^2-\sigma_X^2-\Delta\lambda&0\\
0&0&\sigma_Z^2-\sigma_X^2-\Delta\lambda
\end{array}
\right]
\left[\begin{array}{c}
0\\
\Delta\mu_Y\\
\Delta\mu_Z
\end{array}\right]\\
&=
\frac{-1}{\sqrt{n}}
\left[
\begin{array}{c}
\sigma^{2}_{X}N_{11}\\
\sigma_X\sigma_Y N_{21}\\
\sigma_X\sigma_Z N_{31}
\end{array}
\right]+
\left[
\begin{array}{c}
\Delta\lambda\\
0\\
0\\
\end{array}
\right]\\
&-
\frac{1}{\sqrt{n}}\left[
\begin{array}{ccccccccc}
0&0&0\\
0&\sigma^{2}_{Y} N_{22}&\sigma_Y\sigma_Z N_{23}\\
0&\sigma_Y\sigma_Z N_{32}&\sigma^{2}_{Z}N_{33}
\end{array}\right]
\left[
\begin{array}{c}
0\\
\Delta\mu_Y\\
\Delta\mu_Z
\end{array}
\right]\\
&-\frac{1}{\sqrt{n}}\left[
\begin{array}{ccccccccc}
0&\sigma_X\sigma_Y N_{12}&\sigma_X\sigma_Z N_{13}\\
0&0&0\\
0&0&0
\end{array}\right]
\left[
\begin{array}{c}
0\\
\Delta\mu_Y\\
\Delta\mu_Z
\end{array}
\right]
\end{align*}
the above equation for matrices can be separated into two parts. First the single equation for $\Delta\lambda$:
\begin{equation}
\label{Deltalambda}
\Delta\lambda=\frac{1}{\sqrt{n}}\sigma^{2}_{X} N_{11}
+\frac{\sigma_X}{\sqrt{n}}(\sigma_Y N_{12}\Delta\mu_Y+
\sigma_Z N_{13}\Delta\mu_Z),
\end{equation}
which we will use to determine $\Delta\lambda$. Then the $p-1$ dimensional equation for $\Delta\vec{\mu}$ given as follows:
\begin{align*}&\left[
\begin{array}{ccccccccc}
\sigma_Y^2-\sigma_X^2-\Delta\lambda&0\\
0&\sigma_Z^2-\sigma_X^2-\Delta\lambda
\end{array}
\right]
\left[\begin{array}{c}
\Delta\mu_Y\\
\Delta\mu_Z
\end{array}\right]\\
&=
\frac{-1}{\sqrt{n}}
\left[
\begin{array}{c}
\sigma_X\sigma_Y N_{21}\\
\sigma_X\sigma_Z N_{31}
\end{array}
\right]-
\frac{1}{\sqrt{n}}\left[
\begin{array}{cc}
\sigma_{Y^2}N_{22}&\sigma_Y\sigma_Z N_{23}\\
\sigma_Y\sigma_Z N_{32}&\sigma_{Z^2} N_{33}
\end{array}\right]
\left[
\begin{array}{c}
\Delta\mu_Y\\
\Delta\mu_Z
\end{array}
\right]
\end{align*}

If $\Delta\lambda$ is given, we can solve the above equation for $\Delta\vec{\mu}=(\Delta\mu_Y,\Delta\mu_Z)$ to find:
\begin{align*}
&\Delta\vec{\mu}=\left[
\begin{array}{c}
\Delta\mu_Y\\
\Delta\mu_Z
\end{array}\right]\\
&=
\left(I-
\frac{-1}{\sqrt{n}}\left[
\begin{array}{cc}
\frac{1}{\sigma_Y^2-\sigma_X^2-\Delta\lambda}&0\\
0&\frac{1}{\sigma_Z^2-\sigma_X^2-\Delta\lambda}
\end{array}
\right]
\cdot
\left[
\begin{array}{cc}
\sigma_{Y^2} N_{22}&\sigma_Y\sigma_Z N_{23}\\
\sigma_Y\sigma_Z N_{32}&\sigma_{Z^2} N_{33}
\end{array}\right]\right)^{-1}\\
&\cdot
\frac{-1}{\sqrt{n}}
\left[
\begin{array}{c}
\frac{\sigma_X\sigma_Y}{\sigma_Y^2-\sigma_X^2-\Delta\lambda} N_{21}\\
\frac{\sigma_X\sigma_Z}{\sigma_Z^2-\sigma_X^2-\Delta\lambda} N_{31}
\end{array}
\right]
\end{align*}
where $I$ is the identity matrix. Now, let $D_1$ be the matrix
\begin{align}
D_1:=-\left[
\begin{array}{cc}
\frac{1}{\sigma_Y^2-\sigma_X^2-\Delta\lambda}&0\\
0&\frac{1}{\sigma_Z^2-\sigma_X^2-\Delta\lambda}
\end{array}
\right]
\end{align}
and let $E_1$ be the matrix obtained from $E$ by deleting the first column and the first row:
\begin{align}
E_1=
\left[
\begin{array}{cc}
E_{22}&E_{23}\\
E_{32}&E_{33}
\end{array}\right]=\frac{1}{\sqrt{n}}
\left[
\begin{array}{cc}
\sigma^{2}_{Y}N_{22}&\sigma_Y\sigma_Z N_{23}\\
\sigma_Y\sigma_Z N_{32}&\sigma^{2}_{Z} N_{33}
\end{array}\right]
\end{align}

Therefore, we have
\begin{align}
\Delta\vec{\mu}=
\left[
\begin{array}{c}
\Delta\mu_Y\\
\Delta\mu_Z
\end{array}\right]= -\frac{1}{\sqrt{n}}(I-D_1E_1)^{-1}
\left[
\begin{array}{c}
\frac{\sigma_X\sigma_Y}{\sigma_Y^2-\sigma_X^2-\Delta\lambda}N_{21}\\
\frac{\sigma_X\sigma_Z}{\sigma_Z^2-\sigma_X^2-\Delta\lambda}N_{31}
\end{array}
\right]
\end{align}
Now when the spectral norm of $D_1E_1$ is less than $1$, then we get the formula:
$$(I-D_1E_1)^{-1}=I+D_1E_1+(D_1E_1)^2+(D_1E_1)^3+\ldots$$
In the case where $D_1E_1$ has spectral norm quite a bit less than $1$, we can approximate $(I-D_1E_1)^{-1}$  by $I$ and find that:
\begin{align*}
\Delta\vec{\mu}=
\left[
\begin{array}{c}
\Delta\mu_Y\\
\Delta\mu_Z
\end{array}\right]
\approx -\frac{1}{\sqrt{n}}
\left[
\begin{array}{c}
\frac{\sigma_X\sigma_Y}{\sigma_Y^2-\sigma_X^2-\Delta\lambda}\hat{N}_{21}\\
\frac{\sigma_X\sigma_Z}{\sigma_Z^2-\sigma_X^2-\Delta\lambda}\hat{N}_{31}
\end{array}
\right]
\end{align*}
with the relative error in that approximation being less than $\frac{|D_1E_1|}{1-|D_1E_1|}$. We can now plug the above formula into \ref{Deltalambda} and get:
\begin{equation}
\label{babal}
\Delta\lambda\approx \frac{1}{\sqrt{n}}\sigma^{2}_{X} N_{11}
-\frac{\sigma^{2}_X}{\sqrt{n}}(\frac{\sigma_Y^2 N_{12}^2}{\sigma_Y^2-\sigma_X^2-\Delta\lambda}+
\frac{\sigma_Z^2 N_{13}^2}{\sigma_Z^2-\sigma_X^2-\Delta\lambda}),
\end{equation}
If we don't take a three dimensional vector $\vec{X}=(X,Y,Z)$ but instead a $p$-dimensional $\vec{X}=(X_1,X_2,\ldots,X_p)$ with independent normal entries, we find that the formula \ref{babal} becomes
\begin{equation}
\label{babal2}
\Delta\lambda= \frac{1}{\sqrt{n}}\sigma^{2}_{X_1} N_{11}
- \frac{\sigma _{X_1}^2}{\sqrt{n}}\left(\frac{\sigma_{X_2}^2 N_{12}^2}{\sigma_{X_2}^2-\sigma_{X_1}^2-\Delta\lambda}+
\frac{\sigma_{X_3}^2 N_{13}^2}{\sigma_{X_3}^2-\sigma_{X_1}^2-\Delta\lambda}+
\ldots+
\frac{\sigma_{X_p}^2 N_{1p}^2}{\sigma_{X_p}^2-\sigma_{X_1}^2-\Delta\lambda}\right)
\end{equation}
where $$N_{1j}=\sqrt{n}\frac{E_{1j}}{\sigma_{X_1}\sigma_{X_j}}$$

Also, here $\sigma_{X_1}^2+\Delta\lambda$ represents the eigenvalue of the sample covariance, to which we compare
the first eigenvalue $\sigma_1^2$ of the true covariance. Note that we did not use the order of eigenvalues,
hence $\sigma_1^2$ could be any eigenvalue of the true covariance.
Also, for our formula \ref{babal2} we don't need the fact that the ground truth eigenvalues are ordered to be held. So, to simplify notation let us denote by $\sigma^2_j$ the variance $VAR[X_j]=\sigma^2_{X_j}$ which is also the $j$-the eigenvalue of the covariance matrix
$COV[\vec{X}]$. Now, $\sigma_{X_1}^2+\Delta\lambda$ can denote any eigenvalue of the sample covariance matrix. So, let
us write $$\hat{\sigma}^2_1>\hat{\sigma}^2_2>\ldots>\hat{\sigma}^2_p$$
for the eigenvalues of the sample covariance.
Let $i^*$ be the index of the sample covariance eigenvalue to which we compare the eigenvalue $\sigma^2_i$.
Hence, with this generalisation \ref{babal2} becomes:
\begin{equation}
\label{Deltalambda_2}
\Delta\lambda_i=\sigma^2_{i^*} -\sigma^2_{i}
\end{equation}
and \ref{babal2} can be rewritten as
\begin{equation}
\label{i^*}
\Delta\lambda_i\approx-
\frac{\sigma _i^2}{\sqrt{n}}\sum_{j\neq i}\frac{\sigma_{j}^2 N_{ij}^2}{\sigma_{j}^2-\hat{\sigma}^2_{i^*}},
\end{equation}
where we left out the term  $\frac{1}{\sqrt{n}}\sigma^{2}_{X_1} N_{11}$ which is a smaller order term.
Now, for the above \ref{i^*} to be useful, we need
 \begin{equation}
 \label{difference}
 \sigma_{j}^2-\hat{\sigma}^2_{i^*}
 \end{equation}
not to be too small.
Indeed if for example  the sample eigenvalue $\hat{\sigma}^2_{i^*}$ is equal to one of the ground truth eigenvalues $\sigma^2_j$ with $j\neq i$,
then we would have zero in the denominator of one of the terms in the sum in \ref{i^*}. The only way to control this is to take $i^*$ to be the index
of the sample eigenvalue which comes closest to $\sigma^2_i$. In this way, we guarantee that in our sum \ref{i^*}, the expression \ref{difference} does not get beneath $0.5$ times the spectral gap number $i$ of the ground truth.
The spectral gap is defined as follows:
\begin{equation}
\label{spectral_gap}
{\tt spectral\_ gap}_i:=min\{\sigma^2_{i-1}-\sigma^2_i,\sigma^2_{i}-\sigma^2_{i+1}\}
\end{equation}
and so we take $i^*$ to be defined as:
\begin{equation}
\label{deffi*}
i^*:=\argmin_j(j\mapsto|\sigma_i^2-\hat{\sigma}^2_j|).
	\end{equation}
The other reason for taking such $i^*$ is that the matrix $D_1$ defined before may also explode due to \ref{difference} being uncontrolled small. Now, defining $\Delta\lambda_i$ using $i^*$ leads t a meaningless formula in the $O(1)$ part of the spectrum, which has $O(p)$ eigenvalues. So their distances should be about $O(1/p)$. Then, if we choose to compare $\sigma^2_i$ with the sample eigenvalues that come closest, we get $\Delta\lambda_i$ to be meaningless: the difference of sample spectrum and ground truth is $O(1)$ in that area. However, since the eigenvalues build a continuum, their distances are infinitesimal, which means we get
$\Delta\lambda_i$ defined in \ref{Deltalambda} to be infinitesimal and not $O(1)$. So, the current section is for an error of the spectrum larger than $O(p^{0.5})$. (Compare with the remark at the very  end of this section). The case of the eigenvalues of order $O(1)$ is treated in the next section.
	
	Now, we assume that the ground truth eigenvalues are spaced very regularly. So that
	$$\sigma^2_i-\sigma^2_{i+1},\sigma^2_{i}-\sigma^2_{i+2},\sigma^2_{i}-\sigma^2_{i+3},\ldots$$
	behave like the sequence
	$\Delta_i,2\cdot\Delta_i,3\cdot \Delta_i,\ldots$, where $\Delta_i>0$ is the spectral gap defined in \ref{spectral_gap}. We assume the same thing to be held
	on the left side of $i$.  Note that we have
	the series
	$$\frac{1}{\Delta_i}+\frac{1}{2\Delta_i}+\frac{1}{3\Delta_i}+\ldots=\infty$$
is divergent. This implies that terms
$$\frac{\sigma_{i}^2}{\sigma_{i}^2-\hat{\sigma}^2_{i^*}},\frac{\sigma_{i+1}^2}{\sigma_{i+1}^2-\hat{\sigma}^2_{i^*}},\frac{\sigma_{i+2}^2}{\sigma_{i+2}^2-\hat{\sigma}^2_{i^*}},\ldots$$
behave like the terms of a  divergent series. Hence none of its terms dominates the sum
$$\sum_{j\geq i}\frac{\sigma_{i}^2}{\sigma_{j}^2-\hat{\sigma}^2_{i^*}}$$
This has a practical importance for the expression on the right of approximation \ref{i^*}: the expected value
dominates the fluctuation. This means that we can replace the standard normal random variables square $\mathcal{N}^2_{ij}$  by their expectation $1$
and this causes only a smaller order change. The reason is as follows:
consider a sum
\begin{equation}
\label{aN}
\sum_ja_jN_j^2
\end{equation}
where $N_j^2$'s are independent standard normals squared and the $a_j$'s are constants.
Then the expectation of \ref{aN} is $\sum_ja_j$ and dominates the sums standard deviation
as soon as the sum $\sum_j a_j$ dominates any of its terms $a_j$ Take now $a_j$ to be
$\frac{\sigma_{j}^2}{\sigma_{j}^2-\hat{\sigma}^2_{i^*}}$ . The condition that none of the $a_j$ dominates
the sum is satisfied due to the series being divergent, We act also as if $\hat{\sigma}^2_{i^*}$
would not be random. Hence, in the sum \ref{i^*} we can replace the standard normal square $N_{Ij}^2$
by their expected value $1$ and this will only cause a smaller order change. Hence, given the condition\ref{condition_sample_size}, we finally obtain our result by replacing in \ref{i^*} $\mathcal{N}^2_{ij}$ by $1$
to  get the approximation formula:

\begin{equation}
\label{blue}
\boxed{{\color{blue}\Delta\lambda_i\approx
		-\frac{\sigma_i^2}{n}\cdot \sum_{j\neq i}^p
		\frac{\sigma^2_j}{\sigma_j^2-\hat{\sigma}_{i^*}^2}  }}.
\end{equation}
So, this is our result. To prove it, we used the approximation
\begin{equation}\label{thisequation}I\approx(I-D_1E_1)^{-1}\end{equation}
so, we can write
\begin{align}
\label{ID}
(I-D_1E_1)^{-1} =& D_1^{-0.5}\left(I-D_1^{0.5}E_1D^{0.5}\right)^{-1}D_1^{0.5}\nonumber\\
=& D_1^{-0.5}(I+D_1^{0.5}E_1D_1^{0.5}+(D_1^{0.5}E_1D_1^{0.5})^2+\ldots)D_1^{0.5}
\end{align}
where $D_1^{0.5}$ designates the square root of the matrix $D_1$ obtained
by taking all the eigenvalues and replacing them by their square root.
We also use the geometric series development for the last equation above:
$$(I-D_1^{0.5}E_1D^{0.5})^{-1}=I+D_1^{0.5}E_1D_1^{0.5}+(D_1^{0.5}E_1D_1^{0.5})^2+\ldots$$
which is valid as soon as $D_1^{0.5}E_1D_1^{0.5}$ has all eigenvalues strictly smaller than
$1$ in absolute value. If all these eigenvalues have their absolute values much smaller than
$1$, then we can use the approximation:
$$I\approx I+D_1^{0.5}E_1D_1^{0.5}+(D_1^{0.5}E_1D_1^{0.5})^2+\ldots$$
replacing the expression on the right side of the equation above by $I$ this into the right side of \ref{ID}, we find:
\begin{align*}
&D_1^{-0.5}(I+D_1^{0.5}E_1D_1^{0.5}+(D_1^{0.5}E_1D_1^{0.5})^2+\ldots)D_1^{0.5}\\
&\approx D_1^{-0.5}\cdot I\cdot D_1^{0.5}\\
&=I
\end{align*}
which with the help of \ref{ID} leads
$$(I-D_1E_1)^{-1}\approx I$$
and our \ref{thisequation}. So, this is the last thing remaining to be proven in order to establish \ref{blue}.
Again, we need $D_1^{0.5}E_1D_1^{0.5}$ to have spectral norm close to zero, which is the same
as looking at the spectral norm of $|D_1|^{0.5}E_1|D_1|^{0.5}$,
where the matrix $|D_1|$is obtained from $D_1$ by replacing the eigenvalues by their absolute values.

Again $E_1$ is the matrix obtained from
$$E=\hat{\Cov}[\vec{X}]-\Cov[\vec{X}]$$
by deleting the first row and column. Similarly we take the diagonal matrix with $j$-th entry equal to $\sigma_j^2/(\sigma_{j}^2-\hat{\sigma}_{i^*2}i^2)$ and then delete the first row and column to obtain $D_1$ from the finite dimensional approximation. However, we do not attempt to bound $|D_1E_1|$ in what follows. Rather, we work with bounding the spectral norm of the matrix $|D_1|^{\frac{1}{2}}\cdot E_1\cdot |D_1|^{\frac{1}{2}}$. Let us go back to the three dimensional case, which is good to visualise what is going on:
\begin{align*}
&|D_1|^{\frac{1}{2}}E_1 |D_1|^{\frac{1}{2}}\\
&=\frac{1}{\sqrt{n}}\left[
\begin{array}{cc}
\frac{1}{\sqrt{|\sigma_Y^2-\sigma_X^2-\Delta\lambda|}}&0\\
0&\frac{1}{\sqrt{|\sigma_Z^2-\sigma_X^2-\Delta\lambda|}}
\end{array}
\right]\cdot
\left[
\begin{array}{cc}
\sigma_{Y^2}N_{22}&\sigma_Y\sigma_Z N_{23}\\
\sigma_Y\sigma_Z N_{32}&\sigma_{Z^2} N_{33}
\end{array}\right]\\
&\cdot
\left[
\begin{array}{cc}
\frac{1}{\sqrt{|\sigma_Y^2-\sigma_X^2-\Delta\lambda|}}&0\\
0&\frac{1}{\sqrt{|\sigma_Z^2-\sigma_X^2-\Delta\lambda|}}
\end{array}
\right]\\
&=
\frac{1}{\sqrt{n}}\left[
\begin{array}{cc}
\frac{\sigma_Y}{\sqrt{|\sigma_Y^2-\sigma_X^2-\Delta\lambda|}}&0\\
0&\frac{\sigma_Z}{\sqrt{|\sigma_Z^2-\sigma_X^2-\Delta\lambda|}}
\end{array}
\right]\cdot
\left[
\begin{array}{cc}
N_{22}&N_{23}\\
N_{32}& N_{33}
\end{array}\right]\\
&\cdot\left[
\begin{array}{cc}
\frac{\sigma_Y}{\sqrt{|\sigma_Y^2-\sigma_X^2-\Delta\lambda|}}&0\\
0&\frac{\sigma_Z}{\sqrt{|\sigma_Z^2-\sigma_X^2-\Delta\lambda|}}
\end{array}\right]
\end{align*}

Recall that $E$ is the error matrix when estimating the covariance matrix:
$$E=\hat{\Cov}[\vec{X}]-\Cov[\vec{X}]$$
and $E_1$ is obtained from $E$ by deleting the first row and columns from $E$.
So, we have $E_1$ is equal to
\begin{equation}
\label{matrix_formula}
E_1=
\frac{1}{\sqrt{n}}\left[
\begin{array}{cc}
\sigma_Y&0\\
0&\sigma_Z
\end{array}
\right]\cdot
\left[
\begin{array}{cc}
N_{22}&N_{23}\\
N_{32}& N_{33}
\end{array}\right]
\cdot\left[
\begin{array}{cc}
\sigma_Y&0\\
0&\sigma_Z
\end{array}\right]
\end{equation}

We note the similarity between the formula for the covariance matrix estimation error given in \ref{matrix_formula} and the formula , for
\begin{equation}
\label{DED}
|D_1|^{\frac{1}{2}}E_1 |D_1|^{\frac{1}{2}}.
\end{equation}
This shows that the matrix \ref{DED} can be interpreted as a covariance estimation matrix, but with the eigenvalues
of the covariance not being to
$$\frac{\sigma_Y}{\sqrt{|\sigma_Y^2-\sigma_X^2-\Delta\lambda|}}$$
and
$$\frac{\sigma_Z}{\sqrt{|\sigma_Z^2-\sigma_X^2-\Delta\lambda|}}$$
Now, let us go back to the $p$ dimensional case. Similarly, we get that matrix \ref{DED} is the covariance estimation error matrix,
when the ground truth eigenvalues are:
$$\frac{\sigma^2_1}{(\sigma^2_1-\hat{\sigma}^2_{i^*})^2},\frac{\sigma^2_2}{(\sigma^2_2-\hat{\sigma}^2_{i^*})^2},\ldots,
\frac{\sigma^2_{i-1}}{(\sigma^2_{i-1}-\hat{\sigma}^2_{i^*})^2},\frac{\sigma^2_{i+1}}{(\sigma^2_{i+1}-\hat{\sigma}^2_{i^*})^2},
\ldots,\frac{\sigma^2_p}{(\sigma^2_p-\hat{\sigma}^2_{i^*})^2}$$
	where we act as if $\hat{\sigma}^2_{i^*}$ would be non-random.

We can figure out the spectral norm of $D_i E_i$ up to a universal constant thanks to the break through result
of Koltschinksi an Klounici \cite{koltchinskii2017normal}. They show that for an estimated covariance matrix, the spectral norm of the error matrix $|E|$ is typically bounded by
\begin{equation}
\label{boundC}
|E|\leq C\cdot \left(\max_{j}\frac{\sigma_j}{\sqrt{n}}\right)\sqrt{\sum_j\sigma_j^2}
\end{equation}
where $C>0$ is a universal constant, which does not depend on $n$ or on the sequence $\sigma_1^2,\sigma_2^2,\ldots,\sigma_p^2$ of ground truth eigenvalues. They get a hard edge sharply exponential decaying property for the probability to have eigenvalues bigger than the bound\ref{boundC}.  So, we can apply the formula of koltschinksi and Lounici to
our matrix \ref{DED}, since that matrix is also covariance error matrix. For this we need to replace $\sigma_j$ by  $\frac{\sigma_j}{\sqrt{|\sigma_j^2-\hat{\sigma}_{i^*}^2|}}$ for every $j\neq i$ in our bound \ref{boundC}. This gives
us a tied bound, which typically holds for the spectral norm of $|D_i|^{0.5}E_i|D_i|^{0.5}$: So, with high probability:
$${\tt spectral\; norm\;of}\;|D_i|^{0.5}E_i|D_i|^{0.5}\leq
C\cdot\left(
\frac{1}{\sqrt{n}}\max_{j\neq i}\frac{\sigma_j}{\sqrt{|\sigma_j^2-\hat{\sigma}_{i^*}^2|}}\right)\cdot
\sqrt{\sum_{j\neq i} \frac{\sigma^2_j}{|\sigma_j^2-\hat{\sigma}_{i^*}^2|}
}$$
Noting that for fixed $i$, expression
$$\frac{\sigma_j^2}{|\sigma_j^2-\sigma_i^2|}=\frac{1}{|1-(\frac{\sigma_i^2}{\sigma_j^2})|}$$
becomes smaller as $\sigma_j^2$ moves away from $\sigma_i^2$ in both directions,
we get that the maximum is about:
$$\max_{j\neq i}\frac{\sigma_j^2}{|\sigma_j^2-\sigma_i^2|}\approx
\frac{\sigma_i^2}{{\tt spectral\;gap_i}}$$
so that we get
\begin{equation}
\label{condition4secondapprox}{\tt spectral\; norm\;of}\;|D_i|^{0.5}E_i|D_i|^{0.5}\leq
C\cdot
\frac{\sigma_i}{\sqrt{n}\sqrt{{\tt spectral \; gap}_{i^*}}}\cdot
\sqrt{\sum_{j\neq i} \frac{\sigma^2_j}{|\sigma_j^2-\hat{\sigma}_{i^*}^2|}
}
\end{equation}
So for approximation \ref{blue} to hold up to a smaller error term, we simply need the approximation \ref{thisequation}. On the other hand, for \ref{thisequation} to hold,
we need the spectral norm of $|D_i|^{0.5}E|D_i|^{0.5}$ to be close to $0$. To guarantee this, we can use condition \ref{condition4secondapprox}. So, we need the right side of \ref{condition4secondapprox} to be quite a bit below $1$. Formally we want a  small constant $\epsilon$ so that $0<\epsilon<1$ and \ref{condition4secondapprox} is less than $\epsilon$, which gives the condition on the sample size $n$:
\begin{equation}
\label{condition_sample_size}\sqrt{n}\geq C\cdot
\frac{\sigma_i}{\epsilon \sqrt{{\tt spectral \; gap}_{i^*}}}\cdot
\sqrt{\sum_{j\neq i} \frac{\sigma^2_j}{|\sigma_j^2-\hat{\sigma}_{i^*}^2|}
}
\end{equation}
This typically will hold, for the eigenvalues of order
bigger or equal to $O(p^{\frac{1}{2}})$ assuming the eigenvalues $\sigma^2_j$ to be regularly spaced.

So, we assume standardized data, which means $\sum_j\sigma^2_j=p$.
Hence, if a certain type of eigenvalues has a sum less than $O(p)$,
they would be not relevant. Hence if we consider eigenvalues of order $O(p^\beta)$ with $0<\beta<1$, then there needs to be $O(p^{1-p})$ of them at least, since otherwise their sum would be too small
to play an important role. If they are spaced regularly, then the order of the spectral gap must be $O(p^\beta/p^{1-\beta})=O(p^{2\beta-1})$.
Now, with enough regularity of the eigenvalues, the expression on the right side of \ref{condition_sample_size} is approximately equal to $\frac{\sigma_i^2}{{\tt spectral \; gap}_i}$. Then plugging in the formula $O(p^{2\beta-1})$ for the spectral gap and $O(p^\beta)$ for   $\sigma_i^2$ into \ref{condition_sample_size}, we get that condition
that \ref{condition_sample_size} is satisfied when $\beta>0.5$

\section{The Case of large $c$ for the Sample Size $n = c\cdot p$}
\label{section_large_constant}
The current section is for proof of an approximation formula
for the difference between the spectrum of sample covariance and
 ground truth covariance in the case that the constant $c$
is very large. For this we assume as usual a data matrix
$X$ of dimension $n\times p$, where $n=c\cdot p$ with i.i.d.
normal rows with expectation $0$. Then,
we let $p$ go to infinity. Our approximation formula
is supposed to hold, for large $c$. Again,
let
$$COV[\vec{X}]=\Sigma_p:=E[X^tX]$$ denote the $p\times p$ ground truth covariance matrix, which is also denoted by $\Sigma_p$. We denote by $$\hat{COV}[\vec{X}]=\hat{\Sigma}_p=\frac{X^t\cdot X}{n}$$ the sample covariance matrix.
Again, recall that we denote by $\hat{\sigma}_j^2$ the $j$'th eigenvalue of the sample covariance and by $\sigma_j^2$ the $j$-th eigenvalue of the ground truth covariance. In previous cases, we ordered the eigenvalues in decreasing order. The goal of this section, is to show that the approximation
\begin{equation}
\label{approximation}
\hat{\sigma}_i^2-\sigma^2_i\approx  \frac{\sigma^2_i}{n}\sum_{s\notin J^k_i}\frac{\hat{\sigma}^2_s}{\hat{\sigma}^2_s-\hat{\sigma}^2_i},
\end{equation}
holds given $c$ is sufficiently large. The interval $J^k_i=[i-k,i+k]$ is defined so that the sum
approximate the improper integral that is so that:
\begin{equation}
\label{park}
\frac{1}{p}\sum_{s\notin J^k_i}\frac{\hat{\sigma}^2_s}{\hat{\sigma}^2_s-\hat{\sigma}^2_i}\approx \int \frac{x}{x-\hat{\sigma}_i^2}dF^W(x).
\end{equation}
(Note that for i.i.d variables $Y_1,Y_2,\ldots,Y_p$ with a density $f_Y$, in order to have the approximation
$$\frac{1}{p}\sum_s \frac{Y_s}{Y_s-z}\approx \int  \frac{x}{x-z} f_y(x)dx$$
we need to remove a few of the $Y$'s closest to $z$.....)
Here are some detailed explanations. Again, $c$ is constant. The idea is that if we take $c$ really large, but then keep it constant whilst $p$ goes to infinity. $\hat{\sigma}_i^2-\sigma^2_i$ scales like $1/c$.
So, we define $a$ to be:
$$a:=c\cdot (\hat{\sigma}^2_i-\sigma_i^2)$$
We view $a$ as a function of $\sigma^2_i$ or equivalently as a function
of $\hat{\sigma}^2_i$. For $c$ large enough, $a$ should not change a lot and we view it in terms
of $c$ as a constant, which depends on eigenvalue we choose.
Therefore, the goal of this section is to show that for $c$ large enough, $a$ equals the left side of the approximation \ref{park}
up to a small order term,
which would then imply \ref{approximation}.
Instead we are going to prove that $a$ is the right side of \ref{approximation} plus a term $O(\frac{1}{C})$ at the limit after
$p$ goes to $\infty$ holding $c$ fixed\ref{main_result}.

Now we assume that we have the data matrix $X$, which is $n$ times $p$. For the matrix $X$, it has the property that all the columns and rows are independent normal random variables with expectation $0$. More specifically, we assume that there is a normal random vector of length $p$ with independent entries
$$\vec{X}=(X_1,X_2,\ldots,X_p)$$
where
$E(X_j)=0$ for $j=1,2,\ldots,p$ and $X_1,X_2,\ldots,X_p$ are independent.
We assume that they are independent because if we would have data with dependent columns, we could just
change coordinate system and work with principal components and so get independent coordinates.
We assume that $\Var(X_j)=\sigma_j^2$.
Hence, the covariance matrix $\Cov(\vec{X})$ is a diagonal matrix
\begin{equation}
\label{COV}
\Cov(\vec{X})=
\left[\begin{array}{ccccccccc}
\sigma_1^2&0&0&\ldots&0\\
0&\sigma^2_2&0&\ldots&0\\
0&0&\sigma^2_3&\ldots&0\\
\ldots\\
0&0&0&\ldots&\sigma^2_p
\end{array}\right]
\end{equation}
Again, we have the $n\times p$ data matrix $X$:
\begin{align}
X=
\left[\begin{array}{ccccccccc}
X_{11}&X_{12}&X_{13}&\ldots&X_{1p}\\
X_{21}&X_{22}&X_{23}&\ldots&X_{2p}\\
X_{31}&X_{32}&X_{33}&\ldots&X_{3p}\\
\vdots & \vdots & \vdots &\ldots & \vdots\\
X_{n1}&X_{n2}&X_{n3}&\ldots&X_{np}\\
\end{array}\right]
\end{align}
So, the rows of $X$ are each distributed like $\vec{X}$ and independent of each other.
Since $E[\vec{X}]=\vec{0}$, the estimated covariance matrix (sample covariance)
is given by
$$ \hat{\Cov}(\vec{X}):=\frac{X^T X}{n} . $$
Now, we are going to look at the data $X$ without the $i$-th column.
More exactly, we are going to replace the $i$-th column by zeros,
and then compute the sample covariance matrix. This estimated covariance matrix will
be denoted by $\hat{\Cov}(\vec{X})_{SUB}$.
So, we have $X_{SUB}$ is defined by
\begin{align*}
X_{SUB}:=
\left[\begin{array}{ccccccccc}
X_{11}&X_{12}&X_{13}&\ldots&X_{1(i-1)}&0&X_{1(i+1)}&\ldots&X_{1p}\\
X_{21}&X_{22}&X_{23}&\ldots&X_{2(i-1)}&0&X_{2(i+1)}&\ldots&X_{2p}\\
X_{31}&X_{32}&X_{33}&\ldots&X_{3(i-1)}&0&X_{3(i+1)}&\ldots&X_{3p}\\
\vdots & \vdots & \vdots &\ldots & \vdots & 0 & \vdots & \ldots &\vdots\\
X_{n1}&X_{n2}&X_{n3}&\ldots&X_{n(i-1)}&0&X_{n(i+1)}&\ldots&X_{np}\\
\end{array}\right]
\end{align*}
Hence the sample covariance matrix for this
"reduced" data matrix
is given by:
\begin{equation}
\label{reduce_cov}
\hat{\Cov}(\vec{X})_{SUB}=\frac{X_{SUB}^T X_{SUB}}{n}.
\end{equation}
The above estimated covariance matrix has the $i$-th row and $i$-th column
being $0$. Other entries are clearly the same as for the full sample
covariance $X^T X/n$.
Now, one eigenvalue of the reduced sample covariance \ref{reduce_cov} is equal to $0$.
Others are denoted by	
\begin{equation}
\label{old}
\hat{\sigma}^2_{SUB,1}>\hat{\sigma}^2_{SUB,2}>\hat{\sigma}^2_{SUB,3}>\ldots>\hat{\sigma}^2_{SUB,(p-1)}.
\end{equation}
The eigenvalues of the original sample covariance $\hat{\Cov}[\vec{X}]$ are denoted by:
\begin{equation}
\label{new}
\hat{\sigma}^2_1>\hat{\sigma}^2_2>\ldots>\hat{\sigma}^2_p.
\end{equation}
In the lemma \ref{interlacing_lemma}, we show the interlacing property. That is
we always have:
$$\hat{\sigma}^2_1\geq \hat{\sigma}^2_{SUB,1}\geq
\hat{\sigma}^2_2\geq \hat{\sigma}^2_{SUB,2}\geq \ldots\geq
\hat{\sigma}^2_{p-1}\geq \hat{\sigma}^2_{SUB,(p-1)}\geq \hat{\sigma}^2_p
$$
Now, we are going to condition on the data without column $i$, namely condition on $X_{SUB}$. Then, the eigenvalues
$$\hat{\sigma}^2_{SUB,1}>\hat{\sigma}^2_{SUB,2}>\hat{\sigma}^2_{SUB,3}>\ldots>\hat{\sigma}^2_{SUB,(p-1)}.$$
are no longer random. When we add the random $i$-th column to the matrix $X_{SUB}$, new eigenvalues of the full sample covariance matrix, that is
$\hat{\sigma}^2_1>\hat{\sigma}^2_2>\ldots>\hat{\sigma}^2_p$,
become random. We are going to study the evolution of "this particle process".
That is how we get the eigenvalues \ref{new} from \ref{old}.
Now, we denote by $\nu_j$ the eigenvalue $\hat{\sigma}^2_{SUB,j}$,  for all $j=1,2,\ldots,p-1$.
Again, we assume that the spectrum of the ground truth covariance $\Sigma_p=COV[\vec{X}]$ converges to
a limit with distribution function denoted by $F^{\Sigma}$ as $p$ goes to infinity. Also, the empirical distribution of the sample covariance matrix \ref{COV} is denoted by $F^{W_p}$,
whilst the spectrum of the restricted  sample covaraince $\hat{COV}[\vec{X}]_{SUB}$ is denoted by $F^{W_{p-1}}$, where we leave out $0$. We assume $F^{\Sigma_p}$ converges and so $F^{W_p}$ must also converge to a limit $F^W$, so called
Whishard distribution. One can, for example, determine eigenvalues for $\Sigma_p$ by choosing at random i.i.d. from
the distribution $F^{\Sigma}$, which means that we could have that
$\sigma^2_1>\sigma^2_2>\ldots>\sigma^2_p$
as a set obtained by choosing $p$ i.i.d. values from the distribution $F^\Sigma$.
Or one could choose in a more regular way to get faster convergence of $F^{\Sigma_p}$. Now, in our notation $\sigma^2_i$ is the one we leave out. When we add $\sigma^2_i$ to the ground truth spectrum,
we go from the empirical distribution $F^{W_{p-1}}$ to $F^{W_p}$.
Since we have convergence of $F^{W_p}$ to $F^W$, we need
\begin{equation}
\label{distribution_function}
G^p:=p\cdot F^{W_p}-(p-1)\cdot F^{W_(p-1)}
\end{equation} to converge
weakly to
$F^W$, at least when the added eigenvalue $\sigma_i^2$ is chosen at random
from the distribution $F^\Sigma$.

At this stage we are ready to summarize the rest about how we show that our approximation \ref{approximation} holds, for large $c$. Let's look at a few examples first.

EXAMPLE 1:
Assume for example $p=7$, and that we have the spectrum of the restricted covariance $\hat{\Cov}(\vec{X})_{SUB}$
given by:
$$\{\nu_1=7,\nu_2=6,\nu_3=5,\nu_4=4,\nu_5=3,\nu_6=2\}
$$
whilst the full sample covariance 's  $\hat{Cov}(\vec{X})$ spectrum would be:
$$\{\hat{\sigma}^2_1=7,\hat{\sigma}^2_2=6,\hat{\sigma}^2_3=5,{\color{red}\hat{\sigma}^2_4=4.5},\hat{\sigma}^2_5=4,\hat{\sigma}^2_6=3,\hat{\sigma}^2_7=2\}.$$
we see that the difference consists in one point, which has been added. We will denote that point by $\xi$, so in the current example,
we find $\xi=4.5$. In reality it is unlikely that only one points gets added. So, let us look at a more realistic example.

EXAMPLE 2: Again $p=7$ and let spectrum of $\hat{\Cov}(\vec{X})_{SUB}$
be as before ,but the spectrum of the ground truth be changed to:
$$\{\hat{\sigma}^2_1=7,\hat{\sigma}^2_2=6,{\color{red}\hat{\sigma}^2_3=5.5,\hat{\sigma}^2_4=4.5,\hat{\sigma}^2_5=3.5},\hat{\sigma}^2_6=3,\hat{\sigma}^2_7=2\}.$$
In this case, eigenvalues $\nu_1,\nu_2,\nu_5,\nu_6$ are not changed, but all the others are. So we could not view the change as adding one single point. However, we will still do so by viewing the point added $\xi$ to be a random variable with a density function, which is zero outside the interval $[3.5,5.5]$ and centered maybe, in the current case, at $4.5$.  Indeed in that interval the total number of points get increased by one when you go from restricted sample covariance matrix spectrum to full sample covariance.
There are two approaches presented in our research. One is heuristic and maybe easier to understand. It first shows when we take $C>0$ really large, we get a situation like the one presented in the current example: most eigenvalues barely change when we add the additional dimension to go from $\hat{\Cov}(\vec{X})_{SUB}$ to $\hat{\Cov}(\vec{X})$. And the most serious change happens in a restricted interval, which is centered in a certain location.
That location could be viewed as the place where we added a point. The heuristic argument
is then to say that if the additional eigenvalue $\sigma^2_i$ is the $i$-th eigenvalue of the ground truth spectrum, then this should also add a "point" in the $i$-th position of the sample covariance.
If we assume this to be true, one explains in Section \ref{addenda}, that this translates into our formula \ref{approximation} holding up to a small error term.

Now, the approach we pursue in the rest of this Section is obtained by writing down the equation for the distribution of $\xi$. Let us see one more example:

EXAMPLE 3: Take the same restricted spectrum as before, but let the spectrum of the full sample covariance be:
$${\color{red}\{\hat{\sigma}^2_1=7.1,\hat{\sigma}^2_2=6.1,\hat{\sigma}^2_3=5.5,\hat{\sigma}^2_4=4.5,\hat{\sigma}^2_5=3.5,\hat{\sigma}^2_6=2.9,\hat{\sigma}^2_7=1.9\}}$$
So, this time all the eigenvalues are changed a little bit. However, those further from center are changed much less. So, how could we model this as one point $\xi$ added to the spectrum? The answer is that we take the ratio of how much they get moved
to the spectral gap as the probability distribution function. For example, we see, in current example, that between $\nu_2$ and $\hat{\sigma}^2_2$, there is only a distance of $0.1$. So we will assume that the random variable $\xi$, which represents the change in spectrum as one point random variable added, would have a probability of $0.1$ to be to the left of  $\nu_2$.
In other words, we model the probability of $\xi$ by the ratio:
\begin{equation}
	\label{basic}
	P(\xi\leq \nu_j)=E\left[\;\frac{\hat{\sigma}^2_j-\nu_j}{\nu_{j-1}-\nu_j}\;\right]
	\end{equation}
or we should probably take the expectation on the right side of the equation above. If we take the distribution function $G^p$ as defined in \ref{distribution_function}, then at the limit we should get $F^W$. So a microscopic moving average of $G^p$ should converge to $F^W$ as well. Recall that we had defined $a$ to be
$$a=c\cdot(\hat{\sigma}^2_i-\sigma^2_i)$$
The goal is to determine $a$ at the limit when $p$ goes to infinity.
The way to calculate $a$ is as follows. At the limit we know that $\xi$ must have limit distribution $F^W$. Recall that we denote by $\sigma^2_i$ the eigenvalue of the ground truth
covariance matrix. Now, we can add a value chosen at random among
$\sigma_1^2,\sigma_2^2,\ldots,\sigma^{2}_{p-1}$.
In this way, we get a random variable $T$ with probability distribution $F^{\Sigma_{p-1}}$. Then, $\hat{\sigma}^2_i$ is a random variable with distribution $F^{W_{p}}$, which we denote
by $S$ and we get
$$\sigma^2_i=T=S-\frac{a(S)}{c}$$

In order to calculate the value of $a$, what we do in the remainder of this section
is simple: since $\xi$ and $S$ are supposed to have the same probability distribution $F^W$ at the limit,
we write the equation:
\begin{equation}
\label{ref}
P(S\leq x_0)=\int P(\xi\leq x_0|S=s)dF^W(s).
\end{equation}
which is held for every $x_0$. Now, this is one equation and we have one unknown $a$. So we can solve for $a$ given a formula for the conditional probability in the integral on the right side of \ref{ref}. This formula is obtained from an exact formula \ref{boulanger} and \ref{boulanger2} for $\hat{\sigma}^2_j-\sigma^2_i$.  This leads to the approximation \ref{boulanger3}, which holds up to a small order term. And we can rewrite the approximation as:
\begin{equation}
\label{BASIC}
\frac{c\cdot(\hat{\sigma}^2_j-\sigma^2_i+\frac{\sigma^2_i}{c}
\sum_{s\notin J^K_j} \frac{\nu_s}{\nu_s-\hat{\sigma}^2_j})}{\nu_j\cdot \sigma^2_i}\cdot \frac{p}{\nu_{j-1}-\nu_j}\approx
(\nu_{j-1}-\nu_j)\cdot
\sum_{s\notin J^K_j} \frac{\nu_s\mathcal{N}^2_{s}}{\nu_s-\hat{\sigma}^2_j}
\end{equation}
where $\mathcal{N}_1,\mathcal{N}_2,\ldots,\mathcal{N}_{p-1}$ are conditioned
on $\nu_1,\ldots,\nu_{p-1}$ i.i.d. standard normal. Also, the interval $J^K_j=[j-K,j+K]$ is to leave out enough uncontrolled small term in the sum on the right
of \ref{BASIC}, so as to get the sum to be close to the corresponding indefinite integral.

Similar equation to \ref{boulanger3} is given in \cite{amsalu2018recovery}. However,
the novelty of our research is that we understood this equation and \ref{BASIC} is not to determine the macroscopical difference
between sample spectrum and ground truth. Rather it is to determine the evolution of sample viewed as particle process as we add one additional dimension to the data each time step and observe the resulting evolution. Indeed, by the interlacing property we know that
\begin{equation}
\label{interval}
\hat{\sigma}^2_j\in[\nu_{j-1},\nu_j]
\end{equation}
so conditioning on $\nu_1,\nu_2,\ldots,\nu_{p-1}$ the macroscopical position of $\hat{\sigma}^2_j$ is no longer to be determined.
It is its microscopical relative position within the interval on the right of \ref{interval},
which equation \ref{BASIC} allows to determine. By relative microscopical position we mean:
the ratio \begin{equation}
\label{ratio}
\frac{\hat{\sigma}^2_j-\nu_j}{\nu_{j-1}-\nu_j}
\end{equation}
Now note that we can solve equation \ref{BASIC} to determine the value of $\hat{\sigma}^2_j$ inside the interval \ref{interval}.
Also, note that the left side  of \ref{BASIC} is not affected by the exact position of $\hat{\sigma}^2_j$ inside that interval
except for a small order term. Hence, the value of the ratio \ref{ratio} can be viewed as the value of a function $g(.)$
of the left side of \ref{BASIC}. The same thing holds when we take the expectation:
\begin{equation}
\label{expected_ratio}
g\left(\frac{c\cdot(\hat{\sigma}^2_j-\sigma^2_i+\frac{\sigma^2_i}{c}
	\sum_{s\notin J^K_j} \frac{\nu_s}{\nu_s-\hat{\sigma}^2_j})}{\nu_j\cdot \sigma^2_i\cdot f^W(\nu_j)}\right)=
E\left[\frac{\hat{\sigma}^2_j-\nu_j}{\nu_{j-1}-\nu_j}|S\right],
\end{equation}
where we replaced $p/(\nu_{j-1}-\nu_j)$ by the probability density $f^W$ at the limit. We assume that microscopically the adjacent spectral gaps have a joint distribution, which
asymptotically does not depend on location or scale once re-scaled by $\nu_j-\nu_{j-1}$. We know that the conditional probability for $\xi$ is less than $x_0$, $P(\xi\leq x_0|S)$, is given by the expected ratio on the right side of \ref{expected_ratio} according to \ref{basic}. We can thus replace the conditional probability inside the integral on the right side of \ref{ref} by the expression on the left of \ref{expected_ratio}. We would then put $\sigma_i^2=S-a(S)/c$ and put both $\nu_j$ and $\sigma^2_j$ equal to $x_0$
in the expression on the left of \ref{expected_ratio} and solve. This would work if we could determine the function $g(.)$. Alternatively, for large $c$ we do not need to know everything about $g(.)$. Instead, it is enough to know for large $z$ how $g(.)$ behaves. In the present case, we argue that $g(x)\approx 1/|z|$ as long as $z$ is larger in absolute value than a certain constant.

So, this is the method how to determine $a$: we take equation \ref{ref} after plug in the formula given for the conditional probability by \ref{expected_ratio} and solve for $a$.

Next we are going to discuss the detail of it. it turns out that for calculation it is easier to do in two steps: first calculate the change in probability when going from $S$ to $T=S-a(S)/c$, then the change in probability from $T$ to $\xi$. Let us first give one more numerical example, where we can study in details:

EXAMPLE 4: We are dealing with a signed measure.
Assume that $p=6$. And, the  eigenvalues are given as follows:
	$$
	\begin{array}{c|c|c}
	j&\nu_j&\hat{\sigma}^2_j\\\hline
	1&1&0.9\\
	2&2&1.9\\
	3&3&2.5\\
	4&4&3.3\\
	5&5&4.2\\
	6&&5.2\\
	\end{array}$$
	
	So, note that $G^p$ takes the following values:
	$$
	\begin{array}{c||c|c|c|c|c|c|c|c}
	G^p(x)&0&1&0&1&0&1&0&1\\\hline
	x\in&[-\infty,0.9)&[0.9,1)&[1,1.9)&[1.9,2)&[2,2.5)&[2.5,3)&[3,3.3)&[3.3,4)\\
	\end{array}
	$$
	$$
	\begin{array}{c||c|c|c|c}
	G^p(x)&0&1&0&1\\\hline
	x\in&[4.4.2)&[4.2,5]&[5,5.2)&[5.2,\infty]\\
	\end{array}
	$$
	
	First note that due to the interlacing property, we have
	$$\hat{\sigma}^2_1<\nu_1<\hat{\sigma}^2_2<\nu_2\leq\ldots\leq \nu_5\leq \hat{\sigma}^2_6$$
	which implies that the function $G^p$ is alternating between values $0$ and $1$.
		\begin{figure}\centering
     \label{spectrum_comp}
     \includegraphics[width=0.47\columnwidth]{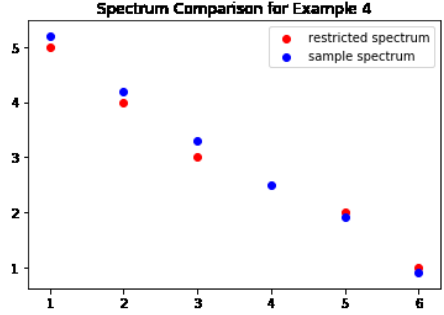}
     \caption{\small Spectrum Comparison for Example 4}
     \end{figure}
	The probability distribution function of a random point is increasing and can not be alternatively
	going up and down like $G^p$ does. When we consider our numerical example, we see that
	$\nu_1$ and $\hat{\sigma}^2_1$ are close to each other. Similarly, in our current example,
	\begin{equation}
	\nu_1\approx \hat{\sigma}^2_1,\nu_2\approx \hat{\sigma}^2_2,\nu_4\approx \hat{\sigma}^2_5,\nu_5\approx \hat{\sigma}^2_6
		\end{equation}
	So, in a very rough approximation we could say that going from spectrum
	\begin{equation}
	\label{spectrum1}
	\{\nu_1,\nu_2,\nu_3,\nu_4,\nu_5\}
	\end{equation}
	to the spectrum
	\begin{equation}
	\label{spec trum2}
	\{\hat{\sigma}^2_1,\hat{\sigma}^2_2,\hat{\sigma}^2_3,\hat{\sigma}^2_4,\hat{\sigma}^2_4,\hat{\sigma}^2_5,\hat{\sigma}^2_6\}
	\end{equation}
	we "add a point in the area $[\nu_2,\nu_4]$". Now, in the interval $[\nu_1,\nu_2]$ the function $G^p$ is 0 on an sub-interval of length $0.9$ and $1$ on a sub-interval of length
	$0.1$. So, on average it is $0.1$ on that interval. Same thing for the interval $[\nu_2,\nu_3]$. In our example, we can write
	\begin{equation}
	\label{01}
	0.1=\frac{\hat{\sigma}^2-\nu_2}{\nu_1-\nu_2}.
	\end{equation}
	so, if we would do a local smoothing, that is a moving average of $G^p$. The value between $\nu_1$ and $\nu_3$ would probably be close to $0.1$. The moving average happens if we re-simulate the situation many times and then take the average.
	Now, when we take a moving average of $G^p$, we would get those values. Again in reality we are interested in a case with very
	large $p$. So instead of $z\mapsto G^p(z)$, we take the map $z\mapsto E[G^p(z)]$, we get a local moving average since we consider values $\nu_1,\nu_2,\ldots,\nu_p$ to be random. But by concentration of measure they fluctuate only microscopically. Hence the moving average will only be microscopical instead of macroscopical. When instead of simulating the data $X$ once,
	we simulate it many times and then build the average of the function $z\mapsto G^p(z)$. For every simulation we get one realisation of $G^p(z)$. That is for every
	$z\in \mathbb{R}$, we get a long term average value for $G^p(z)$ denoted by
	$E[G^p(z)]$. In our example, for $z$ in  $[\nu_1,\nu_2]$,  $E[G^p(z)]$ would probably be close to $0.1$.
	Of course, we need larger $p$ for this work well.
	
So here $z$ is non-random. The formula for
	the value of $E[G^p(z)]$ at $z=\nu_j$ should thus be given by the formula:
	\begin{equation}
	\label{O}
	E[G^p(\nu_j)]=\frac{E[\hat{\sigma}^2_j-\nu_{j}]}{E[\nu_{j-1}-\nu_j]}
	\end{equation}
	Let $f^{W_p}$ denote the probability density of the spectrum of the sample covariance. We average over the distribution function $F^{W_p}$, take the derivative and consider
	$$f^{W_p}(z)=\frac{dE[F^{W_p}(z)]}{dz}.$$
	We can express the expected distance between eigenvalues in function of
	the density function $f^{W_p}$:
	\begin{equation}
	\label{density}
	E[\nu_{j-1}-\nu_j]\approx \frac{1}{p\cdot f^W_p(\nu_j)}.
	\end{equation}
	Recall that by interlacing property, we have that $\hat{\sigma}^2_j$ is in
	$[\nu_{j-1},\nu_{j}]$. The exact location of $\hat{\sigma}^2_j$ is determined by
	an equation. This  equation including the unknown $y$,
	in a slightly simplified form, can be written as:
	\begin{equation}
	\label{y}
	c\cdot \frac{h_j}{\nu_j\cdot \sigma^2_i}=-\frac{1}{p}\left(\frac{1}{\nu_{j-1}-y}+\frac{1}{\nu_{j}-y}\right).
\end{equation}
under the constrain $y\in[\nu_{j-1},\nu_j]$.
The value  $h_j$ is given and we will look at it later. Now assume that $h_j>0$. Then the solution $y$ must be on the right half of the interval
$[\nu_{j-1},\nu_j]$. Assume the length of the interval to be $1/p$, then in that case,
the term $-\frac{1}{p}\left(\frac{1}{\nu_{j-1}-y}\right)$ is at most $0.5$ in absolute value.
So, this leads to the solution of \ref{y} to satisfy
\begin{equation}
\label{nu_j}
y-\nu_j\approx \frac{\nu_j\cdot \sigma^2_i}{p\cdot h_j\cdot c},
\end{equation}
We can now combine  \ref{nu_j}, \ref{density} and \ref{O} to obtain:
\begin{equation}
\label{approximationEG}E[G^p(\nu_j)]\approx\frac{\nu_j\cdot f^{W_p}\cdot \sigma_i^2}{c\cdot h_j}.
\end{equation}
Again, recall that we take the sample size to be equal to $n=c\cdot p$,
where $p$ is the dimension of the space. We  take $c$ very, very large but
it is a fixed constant, whilst  $p$ goes to infinity.
The formula for $h_j$ is given as:
\begin{equation}
\label{hj}
h_j:= \hat{\sigma}^2_j-\sigma_i^2+\frac{1}{c}\cdot \frac{\sigma_i^2}{p}\sum_{s\notin J^K_j}\frac{\nu_s}{\nu_s-\hat{\sigma}^2_j},
\end{equation}
where $J^K_j$ is the integer interval $[j-K,j+K-1]$ and $K$ is a constant of order $O(1)$.
The formula \ref{y} with our choice of $h_j$ given in \ref{hj} is obtained from \ref{boulanger3}, which we prove in the next subsection. We will mention more on that later. Now our goal in this subsection is to show the approximation:
 \begin{equation}
 \label{useful_approx}
 \hat{\sigma}^2_i-\sigma^2_i\approx-\frac{\sigma^2_i}{n}\left(\sum_{s\notin J^K_i}\frac{\nu_s}{\nu_s-\hat{\sigma}^2_i}\right)
 \end{equation}
 to hold "as well as we want" given $c$ large enough. Note that $\hat{\sigma}^2_i-\sigma^2_i$ is going to be of order $O(\frac{1}{C})$.
 So we multiply the left side of \ref{useful_approx} by $1/c$. We define $a$, the re-scaled difference:
 $$a:=c\cdot \left(\hat{\sigma}^2_i-\sigma^2_i\right)$$

 So, to say that the approximation \ref{useful_approx} holds as good as we want given $c$ large enough, would mean
 that the difference between left and right side of \ref{useful_approx} is behaving like $o(\frac{1}{C})$, where $a$ is a $O(1)$ constant different from zero. So, we can treat $a$ like a constant, which is only minimally affected by $c$, but it depends on $\sigma^2_i$.
 
Now, we can also view $j\mapsto h_j$ as a function of $\nu_j$ instead of a function
of the index $j$. This is done by putting
$$h(\nu_j):=h_j.$$

So, let us recapitulate: we add one additional dimension column to the data-matrix $X_{SUB}$. This means that we add one additional eigenvalue $\sigma^2_i$ to the ground truth covariance. The change in spectral distribution due to adding this one dimension is given by the distribution function $G^p$ defined in \ref{distribution_function}. The total value of $G_p$ , which is $G^p[(-\infty,\infty)]=1$. $G^p$ represents a signed measure with positive part having norm $p$ and the negative part having
norm $p-1$. Now, $z\mapsto G^p(z)$ is not yet the distribution function of a random variable, since it is not increasing (it represents how all the points in spectrum get change). But we would like to view the change in spectral measure as one point added, which means instead
of $G^p$ we would like to have the distribution of one random point, i.e. the distribution function of a random variable.
Then we take a local moving average of $G^p(z)$, which corresponds to taking $E[G^p(z)]$. In this way, we obtain a probability distribution function. It means that we can view the change in spectrum as if "one random point $\xi$ was added". We have an exact formula for the probability distribution of $\xi$. So our point is: there exists a symmetric function $g(.)$ around the origin so that if $x_0<b_i$ we have:
\begin{equation}
\label{dothat}
P(\xi<x_0)=E[G^p(x_0)]=g\left(\frac{c\cdot h_j}{\sigma_i^2\cdot x_0\cdot f^W(x_0)}\right).
\end{equation}
and for $x_0>b_i$ we get
$$P(\xi>x_0)=E[G^p(x_0)]=g\left(\frac{c\cdot h_j}{\sigma_i^2\cdot x_0\cdot f^W(x_0)}\right).$$
$b_i$ is the place where function $z\mapsto h(z)$ is zero.
Furthermore for a constant $K$, we have that if $|z|\geq K$, then $g(z)\approx 1/z$. Now we are going to choose the value for $\sigma^2_i$ randomly among all values of $\sigma^2_1,\sigma^2_2,\ldots,\sigma^2_p$. In this way $\hat{\sigma}_i$ is also random. Thus $\hat{\sigma}_i$ is chosen randomly among all the eigenvalues of the sample covariance. Hence, it is a random value. It will be denoted by $S$ ;
$$S:=\hat{\sigma}^2_i$$
and the random variable $S$ has distribution given by $F^{W_p}$. With this we find
$$\sigma^2_i=S-\frac{a(S)}{c}$$
which we define as the variable $T$, so that
$$T=\sigma_i^2=S-\frac{a(S)}{c}.$$
Next we look at $h_j$ given in \ref{hj} and find with our notation:
\begin{equation}
\label{hj2}
h_j=x_0-T\cdot\left(1+\frac{1}{c}\Phi(x_0)\right)
\end{equation}
where $\nu_j$ is denoted by $x_0$ and the function $\Phi(.)$ is the function
$$\Phi(z):=-\frac{1}{p}\left(\sum_{s\notin [z-\epsilon_z,z+\epsilon_z]}\frac{\nu_s}{\nu_s-z}\right)\approx -\int \frac{s}{s-z}f^{W_p}(s)ds$$
 Since now $\hat{\sigma}^2_i$ is random, namely the random variable $S$,
 when $x_0$ is to the left of the $0$ of the function $h(x_0,.)$ we can rewrite equation \ref{dothat}
 as
 \begin{equation}
 \label{dotha2}
 P(\xi<x_0|S)=g\left(\frac{c\cdot (x_0^*-T)}{T\cdot x_0^*\cdot f^W(x_0)}\right).
 \end{equation}
where we replace $h_j$ by the right side of \ref{hj2} and where
$x_0^*$ is defined by
$$x_0^*:=x_0\left(1+\frac{\Phi(x_0}{c})\right)^{-1}$$
similarly for $x_0$ to the right of the zero of $h_j$, we get
\begin{equation}
\label{dotha3}
P(\xi>x_0|T)=g\left(\frac{c\cdot (x_0^*-T)}{T\cdot x_0^*\cdot f^W(x_0)}\right).
\end{equation}
Now at the limit $S$ will have as distribution $F^W$, but the random variable $\xi$ must also have the same distribution at the limit. We can view the process going from $S$ to $\xi$ as a two step
process: first we go from $S$ to $T=S-\frac{a(S)}{c}$ Then, we go from $T$ to $\xi$.
Since $\xi$ and $S$ have same distribution, after going to the limit we have that for a fixed
non-random $x_0$, we must
have equality
$$P(S\leq x_0)=F_S(x_0)=F_\xi(x_0)=P(\xi\leq x_0)$$
when $p$ goes to infinity. So as we go from $S$ to $T$ then go from $T$ to $\xi$, the change in the probability distribution function must cancel out, which is
\begin{equation}
\label{step_zero}
F_T(x_0)-F_S(x_0)=-(F_\xi(x_0)-F_T(x_0))
\end{equation}
Now we assume that $s\mapsto a(s)$ is a continuous function.
Locally it can be considered like a constant. We also assume
it bounded. When we go over from $S$ to $S-\frac{a(S)}{C}$, then locally at $x_0$, this corresponds to a translation of the probability measure of the random variable $S$ by a small distance $\frac{a(x_0)}{C}$. Now, in a small interval of size $\Delta x$,  there is a probability mass equalling approximately the size of the small interval $\Delta x$ times the probability density in that area. So, the amount of probability mass crossing from right to left the point $x_0$ is approximately $f_S(x_0)\cdot \frac{a(x_0)}{C}$. But since the random variable $S$ at the limit has distribution
$F^w$, we get that
\begin{equation}
\label{first_step}
F_T(x_0)-F_S(x_0)\approx f^W(x_0)\frac{a(x_0)}{C}.
\end{equation}
So what is the change due to going over from the variable $T$ to $\xi$? Formula \ref{first_step} shows that there is long distance mass transportation on a scale $\frac{1}{C}$. In other words, the change $F_\xi(x_0)-F_T(x_0)$ is due to the probability mass, which is to the right of $x_0$ under the random variable $T$ and gets to the left of $x_0$ under $\xi$. Then, there is also mass leaving the interval $[-\infty,x_0]$. That is the probability mass which under $T$ is below $x_0$ and after is to the right.
In other words, we get the formula
$$F_\xi(x_0)-F_T(x_0)=\int_{x_0}^\infty   P(\xi<x_0|T=t)       f_T(t)dt-\int_{-\infty}^{x_0}    P(\xi>x_0|T=t)f_T(t)dt.$$
We want to replace the conditional probability on the right side of the last equation above using \ref{dotha2} and
\ref{dotha2}. Then we would get an equation, which is not yet quite right:
\begin{equation}
\label{expression_incorrect}
F_\xi(x_0)-F_T(x_0)=\int_{x_0}^\infty   g\left(\frac{c\cdot (x_0^*-T)}{T\cdot x_0^*\cdot f^W(x_0)}\right )
  f_T(t)dt-\int_{-\infty}^{x_0}    g\left(\frac{c\cdot (x_0^*-T)}{T\cdot x_0^*\cdot f^W(x_0)}\right)f_T(t)dt
\end{equation}
What is the problem with the above?  The problem is that our formulas \ref{dotha2} for $P(\xi\leq x_0|S)$, we need to have $0$ of the function $h(x,S)$ to be to the left of $x_0$. Otherwise, we get $1$ minus the formula. So, we would get
$$P(\xi\leq x_0|S)=1-g\left(\frac{c\cdot (x_0^*-T)}{T\cdot x_0^*\cdot f^W(x_0)}\right )$$
For $S$ and $T$ large enough, this is never going to be the case.
Recall that $h(x,S)$ is defined as
$$h(x,S)=x-(S-\frac{a(S)}{C})\cdot\left(1+\frac{1}{C}\Phi(x)\right)$$
If we set $h(x,S)=0$, it yields the equation
	\begin{equation}
	\label{formule}
	x\cdot (1+\Phi(x))^{-1}=S-\frac{a(S)}{C}
	\end{equation}
	this would yield a zero as a function of $S$: $x(S)$. Now, we want to know when
	that zero of the function $H(x,S) $ is taken at $x_0$. We simply replace
	$x$ by $x_0$ in the formula \ref{formule} and find
\begin{equation}
\label{boundary}
s_0:=x^*+\frac{a(S)}{C}
\end{equation}
So starting at $s_0$ the problem starts and goes until $s=x_0$.
Except that between $x_0$ and $x_0+\frac{a(x_0)}{c}$ the problem is not really there. Because this is the interval, where $S$ is to the right of $x_0$ but the corresponding $T$ is to the left. So, for our calculation $T$ must go to the right and not jump to the left.

In other words, in order to correct \ref{expression_incorrect}, we need to replace the function $g(.)$ by $1-g(.)$ when $S$ is in the interval $[x_0+\frac{a(x_)}{C},x^*+\frac{a(x_0)}{C}]$.
This corresponds to the interval from $x_0$ to $x^*_0$ for $T$. In other words, since formula \ref{expression_incorrect} is written with the integrator $T$ we have to replace the
 function $g(.)$ by $1-g(.)$ on the interval $[x_0,x_0^*]$. This is the same as change of the integration bound from $x_0$ to $x_0^*$ in that formula. It yields the correct formula given as:
 \begin{equation}
 \label{expression}
 F_\xi(x_0)-F_T(x_0)=\int_{x_0^*}^\infty   g\left(\frac{C\cdot (x_0^*-T)}{T\cdot x_0^*\cdot f^W(x_0)}\right )
 f_T(t)dt-\int_{-\infty}^{x_0^*}    g\left(\frac{C\cdot (x_0^*-T)}{T\cdot x_0^*\cdot f^W(x_0)}\right)f_T(t)dt
 \end{equation}

Now, we can take advantage of the symmetries of $g(.)$, and evaluate \ref{expression}:

\begin{align}
&\label{lalala}\int_{x_0^*}^\infty   g\left(\frac{C\cdot (x_0^*-T)}{T\cdot x_0^*\cdot f^W(x_0)}\right )
f_T(t)dt-\int_{-\infty}^{x_0^*}    g\left(\frac{C\cdot (x_0^*-T)}{T\cdot x_0^*\cdot f^W(x_0)}\right)f_T(t)dt\\
&=\int_{x_0^*+\frac{K^*}{C}}^\infty   g\left(\frac{C\cdot (x_0^*-T)}{T\cdot x_0^*\cdot f^W(x_0)}\right )
f_T(t)dt-\int_{-\infty}^{x_0^*-\frac{K^*}{C}}    g\left(\frac{C\cdot (x_0^*-T)}{T\cdot x_0^*\cdot f^W(x_0)}\right)f_T(t)dt\\
&+\int_{x_0^*}^{x_0+\frac{K^*}{C}}  g\left(\frac{C\cdot (x_0^*-T)}{T\cdot x_0^*\cdot f^W(x_0)}\right )
f_T(t)dt-\int_{x_0^*-\frac{K^*}{C}}^{x_0^*}    g\left(\frac{C\cdot (x_0^*-T)}{T\cdot x_0^*\cdot f^W(x_0)}\right)f_T(t)dt\\
\end{align}
where $K^*$ is a constant, which we take sufficiently large so that expression inside the function $g(.)$ is larger in absolute value than $K$ as long as
\begin{equation}
\label{notin}
T\notin[x_0^*-\frac{K^*}{C},x_0^*+\frac{K^*}{C}]
\end{equation}
We can do this because $T$ and $f^W(x_0)$ are supposed to be bounded constants.
so, in other words, we have, when \ref{notin} holds,
$$|\frac{C\cdot (x_0^*-T)}{T\cdot x_0^*\cdot f^W(x_0)}|\geq K$$
However, recall that $K$ is the constant so that for $z$ with $|z|>K$, we have approximately $g(z)=|1/z|$. Hence, when
\ref{notin} holds, we have that
$$g\left(\frac{C\cdot (x_0^*-T)}{T\cdot x_0^*\cdot f^W(x_0)}\right ) \approx \frac{T\cdot x_0^*\cdot f^W(x_0)}{C\cdot (x_0^*-T)}$$
and hence
\begin{align*}&\int_{x_0^*+\frac{K^*}{C}}^\infty   g\left(\frac{C\cdot (x_0^*-T)}{T\cdot x_0^*\cdot f^W(x_0)}\right )
f_T(t)dt-\int_{-\infty}^{x_0^*-\frac{K^*}{C}}    g\left(\frac{C\cdot (x_0^*-T)}{T\cdot x_0^*\cdot f^W(x_0)}\right)f_T(t)dt\\
 & \approx
 \int_{x_0^*+\frac{K^*}{C}}^\infty  \left|  \frac{T\cdot x_0^*\cdot f^W(x_0)}{C\cdot (x_0^*-T)}\right|
 f_T(t)dt-\int_{-\infty}^{x_0^*-\frac{K^*}{C}}  \left| \frac{T\cdot x_0^*\cdot f^W(x_0)}{C\cdot (x_0^*-T)}\right|
  f_T(t)dt\\
  &=\int_{x_0^*+\frac{K^*}{C}}^\infty  \frac{T\cdot x_0^*\cdot f^W(x_0)}{C\cdot (T-x_0^*)}
  f_T(t)dt+\int_{-\infty}^{x_0^*-\frac{K^*}{C}}  \frac{T\cdot x_0^*\cdot f^W(x_0)}{C\cdot (T-x_0^*)}
  f_T(t)dt\\
  & \approx \frac{x_0^*f^W(x_0)}{C} \int_{-\infty}^{\infty}\frac{t}{t-x_0^*}dF_T(t)
  \end{align*}
 which together with \ref{lalala} implies
 \begin{align}
 \label{lalala2}
 \int_{x_0^*}^\infty   g\left(\frac{C\cdot (x_0^*-T)}{T\cdot x_0^*\cdot f^W(x_0)}\right )
 f_T(t)dt-\int_{-\infty}^{x_0^*}    g\left(\frac{C\cdot (x_0^*-T)}{T\cdot x_0^*\cdot f^W(x_0)}\right)f_T(t)dt
 \approx\\
 \frac{x_0^*f^W(x_0)}{C} \int_{-\infty}^{\infty}\frac{t}{t-x_0^*}dF_T(t)\;\;+\;\;O(\frac{1}{C^2})
 \end{align} since
 \begin{equation}
 \label{zero}\int_{x_0^*}^{x_0+\frac{K^*}{C}}  g\left(\frac{C\cdot (x_0^*-T)}{T\cdot x_0^*\cdot f^W(x_0)}\right )
 f_T(t)dt-\int_{x_0^*-\frac{K^*}{C}}^{x_0^*}    g\left(\frac{C\cdot (x_0^*-T)}{T\cdot x_0^*\cdot f^W(x_0)}\right)f_T(t)dt=O(\frac{1}{C^2})
 \end{equation}
 due to the symmetry of $g(.)$. To see why the last order above holds, simply replace $t$ in the numerator above by
 $x_0^*$. Then by symmetry expression \ref{zero} is exactly zero.
 Combining \ref{step_zero}, \ref{first_step}, \ref{expression} and \ref{lalala2}, we find
$$a= -x_0^*\int_{-\infty}^{\infty}\frac{t}{t-x_0^*}dF_T(t)\;\;+\;\;o(\frac{1}{C})$$
 But at the limit as $p$ goes to $\infty$, we have that $T$ has the distribution of the sample spectrum at the limit. Hence, we can replace $F_T$ by $F^W$. Furthermore $x_0$ and $x_0^*$ are at a distance $o(1/C)$ from each other. So replacing $x_0^*$ by $x_0$ only creates a change of order $o(1/C)$. Hence we get
 \begin{equation}
 \label{main_result}
a=-x_0\int_{-\infty}^{\infty}\frac{t}{t-x_0}dF_T(t)\;\;+\;\;o(\frac{1}{C})
 \end{equation}
which is the main result we want to prove. Or rather, we want the approximation \ref{approximation} and instead we proved
the version at the limit after $p$ goes to $\infty$. That version at the limit should imply that the discrete version holds, for $p$ large enough.

\subsection{Derivation for Main formula about the Effect on Eigenvalues of Adding One Dimension}

Now, we are going to change coordinate system. We take the $i$-th canonical vector $\vec{e}_i$ in $\mathbb{R}^p$. And in the orthogonal complement space to $\vec{e}_i$, we take the principal components of
the restricted sample covariance matrix
$$\hat{\Cov}[\vec{X}]_{SUB}.$$
In this way, we notice that the matrix $\hat{\Cov}[\vec{X}]_{SUB}$ and $\hat{\Cov}[\vec{X}]$ are identical
except in the $i$-th column and row. Again,
$\hat{\Cov}[\vec{X}]_{SUB}.$  has its $i$-th row and column containing only $0$'s. Since we use the principal components of $\hat{\Cov}[\vec{X}]_{SUB}$ as basis, it becomes a diagonal matrix. Let us denote that matrix expressed in that basis by $A$:

So,
\begin{equation}
A=\left[\begin{array}{ccccccccc}
\nu_1&0&\ldots&0&0&0&&0&0\\
0&\nu_2&\ldots&0&0&0&&0&0\\
&&&\ldots&&&&&\\
0&0&\ldots&\nu_{i-1}&0&0&\ldots&0&0\\
0&0&\ldots&0&0&0&\ldots&0&0\\
0&0&\ldots&0&0&\nu_{i+1}&\ldots&0&0\\
&&&\ldots&&&&&\\
0&0&\ldots&0&0&0&\ldots&\nu_{p-1}&0\\
0&0&\ldots &0&0&0&\ldots&0&0
\end{array}
\right]
\end{equation}
where for simplicity of notation we denote $\hat{\sigma}^2_{SUB,j}$ by $\nu_j$ for $\forall j=1,2,\cdots,p-1$.
So we have $\nu_1>\nu_2>\ldots>\nu_p\geq 0$
Then we add a $p\times p$ perturbation matrix $E$,
which is zero everywhere except the $i$-th row and $i$-th column
to obtain the full sample covariance matrix $\hat{\Cov}[\vec{X}]$. That is let:
\begin{align*}
E:=\left[\begin{array}{ccccccccc}
0&0&\ldots&0&E_{1i}&0&\ldots&0&0\\
0&0&\ldots&0&E_{2i}&0&\ldots&0&0\\
0&0&\ldots&0&E_{3i}&0&\ldots&0&0\\
&&&&\ldots&&&&\\
E_{i1}&E_{i2}&\ldots&E_{i(i-1)}&E_{ii}&E_{i(i+1)}&\ldots&E_{i(p-1)}&E_{ip}\\
&&&&\ldots&&&&\\
0&0&\ldots&0&E_{(p-2)i}&0&\ldots&0&0\\
0&0&\ldots&0&E_{(p-1)i}&0&\ldots&0&0\\
0&0&\ldots&0&E_{pi}&0&\ldots&0&0\\
\end{array}
\right]
\end{align*}
where $E$ consists of entries of the $i$-th row and column of the matrix $\hat{\Cov}[\vec{X}]$ but expressed in the basis of the principal components of $\hat{\Cov}[\vec{X}]_{SUB}$. So, in that basis, the "full" sample covariance $\hat{\Cov}[\vec{X}]$ is equal to $A+E$. Hence, $E$ is the matrix $\hat{\Cov}[\vec{X}]$. So we have eigenvalues $\nu_1>\nu_2>\ldots>\nu_p\geq 0$
expressed in the basis formed by the principal components of $\hat{\Cov}[\vec{X}]_{SUB}$.

Clearly we have that $E_{ij}=E_{ji}$ for $\forall j\in 1,2,3,\cdots,p$. Furthermore, in Lemma \ref{lemma2}, we prove that except for $E_{ii}$, $E_{ji}$ are independent of each other and normal distributed with expectation $0$ when conditioning on $X_{SUB}$, namely conditioning on the whole data except column $i$. Also, for $j\neq i$, we have that the variance of $E_{ij}$ is equal to $\nu_i\cdot\nu_j/n$. Furthermore $E_{II}\approx \sigma_i^2$.
	
So we have the diagonal matrix $A$ with elements in the diagonal being $\nu_1>\cdots>\nu_{p-1}$ and $0$. These are also the eigenvalues of $A$.

Then we add the perturbation $E$, which only affects the $i$-the column and row. The new eigenvalues are now $\hat{\sigma}_1>\ldots>\hat{\sigma}_p$. there is one more. We are going to calculate these new eigenvalues as a function of the $E_{ij}$'s. To find new eigenvalues we let any of new eigenvalues be denoted by $\lambda+\Delta\lambda$. Therefore, this would be an eigenvalue of $E+A$. Say the corresponding eigenvector is $\vec{\mu}+\Delta\vec{\mu}$, where $\vec{\mu}$ is an eigenvector of $A$.
	
With these notations, we have:
	\begin{equation}
	\label{A+E_1}
	(A+E)(\vec{\mu}+\Delta\vec{\mu})
	=(\lambda+\Delta\lambda)(\vec{\mu}+\Delta\vec{\mu}).
	\end{equation}
Also, since $\vec{\mu}$ is an eigenvector of $A$, we have:
	\begin{equation}
	\label{Amu_1}
	A\vec{\mu}=\lambda\vec{\mu}
	\end{equation}
Subtracting equation \ref{A+E_1} from \ref{Amu_1}, we find:
	\begin{equation}
	\label{noise}
	(A-I\lambda)\Delta\vec{\mu}=-E\vec{\mu}+\Delta\lambda\vec{\mu}+
	-E\Delta\vec{\mu}+\Delta\lambda\Delta\vec{\mu}.
	\end{equation}
	
Now we are going to use \ref{noise} in our case. But to simplify notation, we take $i=1$ and we take a dimension $p=3$. The formula we find will be valid in general. Also, without loss of generality, we can take $\Delta\vec{\mu}$ perpendicular to $\vec{\mu}$. In our present case $\vec{\mu}=(1,0,0)$ is the first eigenvector of the matrix $A$, which is equal to
	\begin{align*}
	A=\left[
	\begin{array}{ccccccccc}
	0&0&0\\
	0&\nu_1&0\\
	0&0&\nu_2
	\end{array}\right]
	\end{align*}
Since $\Delta\vec{\mu}$ is perpendicular to $\vec{\mu}$, we can write $\Delta\vec{\mu}=(0,\Delta\mu_1,\Delta\mu_2)$ Then, we have the perturbation:
	\begin{equation}
	\label{E}
	E=\left[
	\begin{array}{ccccccccc}
	E_{11}&E_{12}&E_{13}\\
	E_{21}&0&0\\
	E_{31}&0&0
	\end{array}\right]
	\end{equation}
	
So, now we can write out equation \ref{noise} with our special case of $A$ and the perturbation matrix $E$ given in \ref{E} to find:
	
	\begin{align*}
	\label{zaza3}
	&\left[
	\begin{array}{ccccccccc}
	0&0&0\\
	0&\nu_1-\lambda-\Delta\lambda&0\\
	0&0&\nu_2-\lambda-\Delta\lambda
	\end{array}
	\right]
	\left[\begin{array}{ccccccccc}
	0\\
	\Delta\mu_1\\
	\Delta\mu_2
	\end{array}\right]
	\\
	&=
	-
	\left[
	\begin{array}{ccccccccc}
	E_{11}\\
	E_{21}\\
	E_{31}
	\end{array}
	\right]+
	\left[
	\begin{array}{ccccccccc}
	\Delta\lambda\\
	0\\
	0\\
	\end{array}
	\right]
	-\left[
	\begin{array}{ccccccccc}
	0&E_{12}&E_{13}\\
	0&0&0\\
	0&0&0
	\end{array}\right]
	\left[
	\begin{array}{ccccccccc}
	0\\
	\Delta\mu_1\\
	\Delta\mu_2
	\end{array}
	\right]
	\end{align*}
the above equation for matrices can be separated into two parts. The first equation gives us an equation for $\Delta\lambda$:
	\begin{equation}
	\label{Deltalambda_1}
	\Delta\lambda=E_{ii}
	+E_{12}\Delta \mu_2+E_{13}\Delta\mu_3
	\end{equation}
	
Then the remaining equation can be used to calculate $\Delta\vec{\mu}$ so that gives:
	
	\begin{align*}
	&\left[\begin{array}{ccccccccc}
	\nu_1-\lambda-\Delta\lambda&0\\
	0&\nu_2-\lambda-\Delta\lambda
	\end{array}
	\right]
	\left[\begin{array}{ccccccccc}
	\Delta\mu_1\\
	\Delta\mu_2
	\end{array}\right]\\
	&=
	-
	\left[
	\begin{array}{ccccccccc}
	E_{12}\\
	E_{13}
	\end{array}
	\right]
	\end{align*}
we can solve the above equation for $\Delta\vec{\mu}$ and then plug into equation \ref{Deltalambda_1} to find:
	$$\Delta\lambda=E_{11}-\frac{E_{12}^2}{\nu_1-\lambda-\Delta\lambda}-
	\frac{E_{13}^2}{\nu_2-\lambda-\Delta\lambda}$$
So far we have given a three dimensional case. But the last formula above is valid in general and becomes:
 \begin{equation}
 \label{dynamic}
 \Delta\lambda=E_{ii}-\sum_{s=1}^{p-1} \frac{E_{si}^2}{\nu_s-(\lambda+\Delta\lambda)},
 \end{equation}

Here $\lambda$ is the eigenvalue of the restricted covariance matrix $A$, which is equal to $0$. So
 $$\lambda=0$$

Furthermore, $\lambda+\Delta\lambda$ is an eigenvalue of the full sample covariance, which is $A+E=\hat{\Cov}[\vec{X}]$. So, in that case $\lambda+\Delta\lambda=\Delta\lambda$ and hence $\Delta\lambda$ represents an eigenvalue of $A+E$. When we consider the equation \ref{dynamic} as an equation of $\Delta\lambda$ assuming other terms are given, we see that for every interval $[\nu_{s-1},\nu_s], \forall s=2,\ldots,{p-1}$, there is one value inside each interval for $\Delta\lambda$ when solving \ref{dynamic}. This is because RHS of \ref{dynamic} is strictly decreasing going from $\infty$ to $-\infty$ as a function of $\Delta\lambda$. So in each interval $[\nu_{s-1},\nu_s]$ for $s=2,\cdots,p-1$, there is exactly one solution to \ref{dynamic}, and that solution is the eigenvalue $\hat{\sigma}_s^2$ of the "full" sample covariance matrix. This is another way to prove the interlacing property proven in Lemma \ref{interlacing_lemma}, that is we have
 		$$\hat{\sigma}^2_1>\nu_1>\hat{\sigma}^2_2>\nu_2>\ldots>\nu_{p-1}>\hat{\sigma}^2_{p}$$
where we recall that $\nu_j=\hat{\sigma}_{j,SUB}^2$ is the $j$-th eigenvalue in decreasing order of
the restricted sample covariance $\frac{X_{SUB}^T\cdot X_{SUB}}{n}$. Assume eigenvalues $\nu_1,\nu_2,\ldots,\nu_{p-1}$ of the restricted covariance are given. Then the equation \ref{dynamic} is the equation, which determines the "dynamix" of the eigenvalues. It shows when we add one eigenvalue in the true covariance matrix, how it is going to affect all the eigenvalues of the sample covariance. We could view this as a particle process, where we add one column after the other to $X$ and have the eigenvalues viewed as particles evolve.
 		
Now, the equation \ref{dynamic} allows to determine all eigenvalues of the full sample covariance. So for example, the $j$-th eigenvalue:
 		\begin{equation}
 			\label{boulanger}
 			\hat{\sigma}^2_j=E_{ii}-\sum_{s=1}^{p-1} \frac{E_{si}^2}{\nu_s-\hat{\sigma}^2_j}
 				\end{equation}
Conditioning on $X_{SUB}$, which is equivalent to condition on all data columns except the $i$-th, the term $E_{si}$ for $s \neq i$ are independent normal random variables with variance equalling:
 			$$ \Var[E_{si}]=\frac{\sigma^2\nu_s}{n}=\frac{\sigma^2_{i}\sigma^2_{SUB,s}}{n} $$
 			
Also,
 			$$E_{ii}=\frac{\sum_j X^2_{ji}}{n}\approx \Var[X_i]=\sigma_i^2$$
Using the last approximation above, we can rewrite equation \ref{boulanger} as
 			\begin{equation}
 			\label{boulanger2}
 			\hat{\sigma}^2_j-\sigma^2_i\approx \; -\frac{\sigma^2_i}{n}\sum_{s=1}^{p-1} \frac{\nu_s\mathcal{N}^2_{s}}{\nu_s-\hat{\sigma}^2_j}
 			\end{equation}
where $\mathcal{N}_s$ are i.i.d. normal random variables conditioned on $X_{SUB}$. By the interlacing property, we have that $\sigma^2_j$ is between $\nu_{j-1}$ and $\nu_j$. Hence, we are considering the interval with natural number close to $j$, that is
 			$$J^k_j:=[j-K,j+K]$$

How big $K$ needs to be will be discussed later. We want all the terms $\nu_t$, which are "micoscopically close" to $\nu_j$, to have their indexes in the interval $J^K_j$. So we can distinguish between terms $\nu_s$ close to $\nu_j$ (and hence to $\hat{\sigma}^2_j$) and others in equation \ref{boulanger2}. For other terms, since terms $\nu_s-\hat{\sigma}^2_j$ are not macroscopically small, we can replace $\mathcal{N}^2_s$ by their expectations $1$ and obtain:
 			$$\sum_{s\notin J^K_j} \frac{\nu_s\mathcal{N}^2_s}{\nu_s-\hat{\sigma}^2_j}\approx
 			\sum_{s\notin J^K_j} \frac{\nu_s}{\nu_s-\hat{\sigma}^2_j}$$
since the expectation dominates the standard deviation.

 Hence we can go back to \ref{boulanger2} to obtain:
 \begin{equation}
 \label{boulanger3}
 \hat{\sigma}^2_j-\sigma^2_i\approx \; -\frac{\sigma^2_i}{n}
 \left(\sum_{s\notin J^K_j} \frac{\nu_s}{\nu_s-\hat{\sigma}^2_j}
 +\sum_{s\in J^K_j} \frac{\nu_s\mathcal{N}^2_{s}}{\nu_s-\hat{\sigma}^2_j}
 \right)
 \end{equation}

 Note that for big eigenvalues the term:
 $$\hat{\sigma}^2_j-\sigma^2_i+\frac{\sigma^2_i}{n}
 \sum_{s\notin J^K_j} \frac{\nu_s}{\nu_s-\hat{\sigma}^2_j}$$ is strongly positive. Hence, when we try to solve approximation \ref{boulanger3} with small $j$, we will have that terms $\sigma^2_1, \sigma^2_2$ are going to be very close to the corresponding $\nu_s$. For large $j$, when $j$ is closer to $p$, we have that $\sigma^2_j$ is going to be close to $\nu_{j-1}$. When the distance is almost indistinguishable close, we get that basically going from the "particles" $\nu_1,\nu_2,\ldots,\nu_{p-1}$ to the "particles"
 $$\hat{\sigma}_1^2,\hat{\sigma}_2^2,\ldots,\hat{\sigma}^2_p$$

we leave the ones in the border unchanged and can say that somewhere in the middle there has been a particle added.

When $n=c \times p$, if the constant $c$ is really big, then most particles don't move except in a small interval. Here we make an Ansatz, which later we can at least heuristically justify: in our system we can add any value for $\sigma_i$, which is the standard deviation of the column that was left out firstly. The values of $X_{SUB}$ are independent of that value and so are the $\nu_1,\ldots,\nu_{p-1}$. So, we can take any value for $\sigma_i$ and see what the outcome is.
\begin{itemize}
	\item Our Ansatz is that (at least when $n=c \times p$ where $c>0$ is large) we have that the particle added due to adding a column with standard deviation $\sigma_i$ should be added in the same relative position as is the position of $\sigma_j^2$ in the original spectrum.
\end{itemize}

So, in other words, if $\sigma^2_i$ is the $i$-th eigenvalue of the original spectrum, then the additional eigenvalue added should also be about in the $i$-position in the sample spectrum, that is to say that we have:
\begin{equation}
\label{thisterm}
\hat{\sigma}^2_j-\sigma^2_i+\frac{\sigma^2_i}{n}
\sum_{s\notin J^K_j} \frac{\nu_s}{\nu_s-\hat{\sigma}^2_j}
\end{equation}
is neither positive nor negative at the point where the particle is added. So according
to our Ansats, that is for $j=i$ and hence compared to the other terms in \ref{boulanger3},
we would have the term \ref{thisterm} be small order. So that for $j=i$, we would have:

\begin{equation}
\label{thisterm}
\hat{\sigma}^2_i-\sigma^2_i+\frac{\sigma^2_i}{n}
\sum_{s\notin J^K_i} \frac{\nu_s}{\nu_s-\hat{\sigma}^2_j}\approx 0
\end{equation}
which implies

\begin{equation}
\label{boulanger4}
\hat{\sigma}^2_i-\sigma^2_i\approx \; -\frac{\sigma^2_i}{n}
\left(\sum_{s\notin J^K_i} \frac{\nu_s}{\nu_s-\hat{\sigma}^2_j}
\right)
\end{equation}
which is the approximation formula we wanted to justify, or rather the continuous version at the limit. For this remember that $\nu_s=\hat{\sigma}_{SUB,s}^2$.

\subsection{Why Big Constant Makes Particles Being Added Locally}
\label{addenda}
let us consider variables $x$ and $y$ in the following equation:
\begin{equation}
\label{heute}
x=-
\frac{\sigma_i}{n}\sum_{s\in J^K_j} \frac{\nu_s\mathcal{N}^2_{s}}{\nu_s-y}
\end{equation}
In the above equation we assume all values given except $x$ and $y$
and we further assume the constrain
\begin{equation}
\label{condition_coco}
y\in [\nu_{j-1},\nu_{j}].
\end{equation}
Note that the function on the RHS of \ref{heute}, seen as a function of $y$, is strictly decreasing going from $\infty$ to $-\infty$ as $y$ goes from $\nu_{j-1}$ to $\nu_j$.
So we can write $y$ as $y(x)$ and there is no ambiguity assuming that we know \ref{condition_coco} to hold.

Now take $x$ to be equal to:
\begin{equation}
\label{heute2}
x= \frac{\nu_{j-1}-\nu_j}{2}-\sigma^2_i +\frac{\sigma^2_i}{n}
\left(\sum_{s\notin J^K_j} \frac{\nu_s}{\nu_s- (\nu_{j-1}+\nu_{j})/2}\right).
\end{equation}

Now, the sum
\begin{equation}
\label{thislabel}
\frac{\sigma^2_i}{n}
\left(\sum_{s\notin J^K_j} \frac{\nu_s}{\nu_s- \hat{\sigma}^2_j}\right)
\end{equation}
is not too much affected by the exact value of $\hat{\sigma}^2_j$ since
$\hat{\sigma}^2_j$ is contained in the interval $[\nu_{j-1},\nu_j]$ and in the sum there  should be no term close to that interval since we take out all the terms, of which index in $J^K_j$. That is we take out all elements, which are microscopically close to that interval. So, in the sum \ref{thislabel}, we can replace $\hat{\sigma}^2_j$
by any point in the interval given in \ref{condition_coco} and should only get a small order change. So we can replace $\hat{\sigma}^2_j$ by the middle of the interval, which is
$(\nu_j-\nu_{j+1})/2$ and still get a similar value. Hence:
\begin{equation}
\label{thislabel2}
\frac{\sigma^2_i}{n}
\left(\sum_{s\notin J^K_j} \frac{\nu_s}{\nu_s- \hat{\sigma}^2_j}\right)
\approx\frac{\sigma^2_i}{n}
\left(\sum_{s\notin J^K_j} \frac{\nu_s}{\nu_s- (\nu_{j-1}+\nu_{j})/2}\right)
\end{equation}

Applying \ref{thislabel2} to \ref{boulanger3}, we obtain
\begin{equation}
\label{boulanger4}
\hat{\sigma}^2_j-\sigma^2_i\approx \; -\frac{\sigma^2_i}{n}
\left(\sum_{s\notin J^K_j} \frac{\nu_s}{\nu_s-(\nu_{j-1}+\nu_{j})/2}
+\sum_{s\in J^K_j} \frac{\nu_s\mathcal{N}^2_{s}}{\nu_s-\hat{\sigma}^2_j}
\right)
\end{equation}
When we replace $\hat{\sigma}_j^2$ by $(\nu_{j-1}-\nu_j)/2$ on the right side of the above approximation, we have:
$$ x\approx-
\frac{\sigma_i}{n}\sum_{s\in J^K_j} \frac{\nu_s\mathcal{N}^2_{s}}{\nu_s-\hat{\sigma}_j^2}$$
The last approximation above shows that we can determine the value of $\hat{\sigma}^2_j$ up to a small error term by solving equation \ref{heute} for $y$ under the constrain \ref{condition_coco} and where $x$ is defined in \ref{heute2}.

When we consider equation \ref{heute} with the constrain \ref{condition_coco}, then: when $x$ is very negative (large absolute value, but negative), then $y(x)$ is close to $\nu_{j}$. On the opposite when $x$  in \ref{heute} is very large positive and condition \ref{condition_coco} holds, then $y(x)$ is close to $\nu_{j-1}$.

Now, assume that $i$ is somewhere in the middle of the spectrum. Then for $j<<i$ we have that $x$
(as given in \ref{heute2}) is positive
and for $j>>i$ we get that $x$ is negative. In order to understand ,let us consider the following: we assume that $n=C\cdot p$ and the constant $C$ is sufficiently
large. Then there is not a big difference between sample spectrum and population spectrum. The difference is still of order
$O(1))$ but has a small constant in front. So, in the first approximation $x$ is about $\sigma^2_j-\sigma^2_i$, which obviously is positive for $j<i$ and negative for $j>i$. 

Next we want to see when $j<<i$, if $x$ is "big enough" to make the solution $y$ of equation \ref{heute} much closer to $\nu_{j}$. Because in that case, we get that $\hat{\sigma}^2_j$ can be approximately found using equation \ref{heute} and that $\hat{\sigma}^2_j$ is also going to be very close to $\nu_{j}$.
So the point $\nu_j$ will be quite indistinguishable of
$\hat{\sigma}^2_j$ for $j<<i$. We want to prove the opposite, when $j>>i$, that $x$ is negative enough so that the solution of equation \ref{heute} is close to $\nu_{j-1}$. This would then imply that $\hat{\sigma}^2_j$ would be very close
to $\nu_{j-1}$. So, in other words, if we can prove these two things: $x$ gets negative enough for $j<<i$ and positive enough fro $j>>i$,
then when we go from the spectrum
\begin{equation}
\label{I}\nu_1>\nu_2>\ldots>\nu_{p-1}
\end{equation}
to
\begin{equation}
\label{II}
\hat{\sigma}^2_1>\hat{\sigma}^2_2>\ldots>\hat{\sigma}^2_p,
\end{equation}
that in principle  for indexes away from $i$, the eigenvalues don't change too much.
In that case, we can view the effect of going from \ref{I} to \ref{II}
as "adding a particle somewhere in the vicinity of $\nu_i$".

What we need for this to work is a regularity of the particles given in \ref{I}. More specifically, assume that:
\begin{equation}
\label{K}
\nu_{j-1}-\nu_{j}\geq K\times\frac{1}{p}
\end{equation}
where $K>0$ is a constant, the interval $J^k_j$ contains only two integers
$J^K=[j-1,j]$. With this we can now rewrite \ref{heute}
as
\begin{equation}
\label{this}
c \cdot x=-
\frac{\sigma_i}{p} \left( \frac{\nu_{j-1}\mathcal{N}^2_{j-1}}{\nu_{j-1}-y}+\frac{\nu_{j}\mathcal{N}^2_{j}}{\nu_{j}-y}\right)
\end{equation}
where we also use that $n=C\dot p,$ where $C>0$ is a constant.

Now we want $y$, the solution of \ref{this}, to be close to $\nu_j$.
What do we mean by this? We may, for example, request that
$|y-\nu_j| $ is at least $10$  times smaller than $\nu_{j-1}-\nu_j$, which means:
$$\frac{|y-\nu_j|}{\nu_{j-1}-\nu_j}\leq  \frac{1}{10}$$
and hence with the help of \ref{K} we find
$$\frac{1}{|y-\nu_j|\cdot p}\geq  \frac{10}{K}$$
which with the help of \ref{this} we obtain
as long as
$$C\cdot x \geq \frac{10}{K}\cdot \frac{\nu_{j}\mathcal{N}^2_{j}}{\sigma_i^2}$$
The inequality above holds with high probability by simply taking the constant $C>0$ large enough since $\sigma_i^2$, $\nu_j$ and $K$ are all of order $O(1)$ and as long as $x$ is not infinitesimal but of order $O(1)$.

Now, say we want for a large (but constant number) $l$, the solution $y$ of equation \ref{this} to be $l$ times closer to $\nu_j$ than $\nu_{i-1}$. Note that this closeness follows from \ref{K}, \ref{this} and
\begin{equation}
\label{cx}
c\cdot x \geq \frac{l}{K}\cdot \frac{\nu_{j}\mathcal{N}^2_{j}}{\sigma_i^2}
\end{equation}
In other words,
\begin{equation}
\label{l}
\frac{|y-\nu_j|}{\nu_{j-1}-\nu_j}\leq  \frac{1}{l}
\end{equation} follows from
\ref{K}, \ref{this} and \ref{cx}. Now, $\hat{\sigma}^2_j$ is the value for $y$ solving
\ref{this} with the contains $y\in[\nu_{j-1}-\nu_j]$. So, if we take $l$ really large (but constant, think
of a million for example), then $y$ becomes almost indistinguishable from $\nu_j$. This means that for practical purpose, $\hat{\sigma}^2_j$ and $\nu_j=\hat{\sigma}^2_{SUB,j}$ will be indistinguishable. This is for $j<<i$.  Similarly for $j>>i$, we can get that $\hat{\sigma}^2_j$ and $\nu_{j-1}=\hat{\sigma}^2_{SUB,j-1}$ will be practically indistinguishable. This means that between the sample covariance and the restricted sample covariance, the difference is mainly in the eigenvalues around the $i$-th, when we add one eigenvalue of size $\sigma_i^2$. Now we need this result to hold uniformly over $i \in 1,2,\ldots,p$. And we
also need this to hold when $j$ is sufficiently close to $i$. What we want is to obtain that if we take the constant $C$ very big, we get that the effect of adding one additional dimension, for practical purposes, does not change the spectrum except in a narrow region of the spectrum around the $i$-th eigenvalue.

As long as $x$ is of $O(1)$, we can obtain this by simple taking
the constant $C$ in \ref{cx} large enough. So we need $x$ to be
bounded below as long as $i$ and $j$ are not too close.

\subsection{Lemma}
In this section we will introduce some lemma that will be used in the following proof.

The first lemma shows the distribution of the restricted covariance matrix.
Recall that $X$ is an $n\times p$ matrix with independent columns, where
entries in column $j$ have standard deviation $\sigma_j$. In order to compute the restricted covariance matrix, firstly we replace the $i$-th column in $X$ by $0$. We denote the new matrix with a zero column as $X_{SUB}$ and the estimated covariance matrix is
\begin{equation}
\label{covariance_SUB}
\hat{\Cov(X)}_{SUB}=\frac{X^T_{SUB}\cdot X_{SUB}}{n}.
\end{equation}
Then we express our restricted covariance matrix \ref{covariance_SUB} in the basis of its principal components. Note that the $i$-th canonical vector $$(0,0,0,\ldots,0,1,0,\ldots,0) \in \mathbb{R}^{p}$$
is a principal component of \ref{covariance_SUB}, which has a $1$ in its $i$-th entries and $0$'s everywhere else. It is the principal component with corresponding eigenvalue $0$. This is because the matrix \ref{covariance_SUB} has its $i$-column and $i$-th row equal to $0$. The principal components are simply eigenvectors by the definition of the principal components. Therefore, when you express a matrix in the basis of its principal components, the matrix becomes diagonal with eigenvalues along the diagonal. In the present case, eigenvalues are denoted by $\hat{\sigma}_{SUB,j}^2$ and also denoted as $\nu_j=\hat{\sigma}_{SUB,j}^2$. So, we are going to represent the full covariance matrix in the basis using principal components of the sub-matrix $\Cov(X)_{SUB}$. The sub-matrix part gets diagonalized in that basis. Except for the $i$-th column and $i$-th row, we are dealing with a diagonal matrix. This is to say that in that basis of eigenvectors of \ref{covariance_SUB}, the full covariance matrix $\frac{X^T X}{n}$ will take the following form:
\begin{equation}
\label{covariancep}
\left[\begin{array}{ccccccccc}
\nu_1&0&\ldots&0&E_{1i}&0&\ldots&0&0\\
0&\nu_2&\ldots&0&E_{2i}&0&\ldots&0&0\\
0&0&\ldots&0&E_{3i}&0&\ldots&0&0\\
&&&&\ldots&&&&\\
E_{i1}&E_{i2}&\ldots&E_{i(i-1)}&E_{ii}&E_{i(i+1)}&\ldots&E_{i(p-1)}&E_{ip}\\
&&&&\ldots&&&&\\
0&0&\ldots&0&E_{(p-2)i}&0&\ldots&0&0\\
0&0&\ldots&0&E_{(p-1)i}&0&\ldots&\nu_{p-2}&0\\
0&0&\ldots&0&E_{pi}&0&\ldots&0&\nu_{p-1}\\
\end{array}
\right]
\end{equation}
The next lemma shows that the non-diagonal entries, that is $E_{si}$ for $s\neq i$ are independent joint normal distributed with given variance.
\begin{lemma}\label{lemma2}
	Assume that we express the sample covariance matrix $\frac{X^T X}{n}$ in the basis
	using principal components of the restricted matrix $\frac{X^T_{SUB} X_{SUB}}{n}$
	to obtain a matrix given in \ref{covariancep}. Then, conditioning on
	$X_{SUB}$, we have that $E_{ij}$ for $j\neq i$ are independent normal distributed
	with
	\begin{equation}
	\label{this}
	\Var[E_{ij}]=\frac{\nu_j\cdot\sigma^2_i}{n}
	\end{equation}
\end{lemma}

\begin{proof}
	Now, assume given an i.i.d. sequence of normal random variables
	$$N_1,N_2,\ldots,N_{p-1}$$
	with expectation $0$ and standard deviation $\sigma$.
	Let
	$$Y:=\sum_j a_j N_j$$ and $$Z=\sum_j b_j N_j$$
	where the $a_j$'s and the $b_j$'s are non-random coefficients.
	Then, $Y$ and $Z$ are jointly normal with covariance:
	\begin{equation}
	\label{COVYZ}
	\Cov(Y,Z)=\sigma^2 \sum_j a_j b_j.
	\end{equation}
	Now, let us look at the sample covariance matrix $\frac{X^T X}{n}$ before the change of basis. If we look at the entry in the $i$-th row and $s$-th column and denote it by $E_{is}^*$ for $s\neq i$. The entry is the product of $i$-th column and $s$-th column of matrix $X$ and divided by $n$. We condition on $X_{SUB}$, which means we condition $i$-th column is a column of i.i.d normal random variables with standard deviation $\sigma_i$. In this situation, we can conclude that the entry $E_{is}^*$ for $s\neq i$ are jointly normal distributed conditioned on $X_{SUB}$. Because these entries are the results of dot product between a vector of coefficients and a vector of i.i.d normal random variables.
	The vector of i.i.d normal random variables is the $i$-th column of $X$.
	According to formula \ref{COVYZ}, in order to find  the covariance
	$$\Cov(E_{is},E_{it})$$
	we need to take the dot product between coefficient vectors.
	Here $E_{si}$ is the dot product of the $s$-th column of $X$(coefficient vector) and the $i$-th column of $X$(random variables) and $E_{it}$ is the dot product of the $t$-th column of $X$(coefficient vector) and the $i$-th column of $X$(random variables). So the covariance
	$\Cov(E_{is},E_{it}$ ) is the product of the $s$-th column times the $t$-th column times $\sigma_i^2$ divided by $n^2$. But this is the $s,t$-th entry of the sample covariance times $\sigma_i^2/n$. In other words conditioning on $X_{SUB}$, the coefficients $E_{is}$ for $s\neq i$, are jointly normal distributed with their covariance matrix equal to product of sample covariance matrix and coefficient $\sigma^2_i/n$ . Now, when you change for a normal vector the basis and take the principal component as a basis, you get a normal vector with independent components and where the variance of the components are the eigenvalues of the original covariance matrix. In our case, these variances are $\nu_s\frac{\sigma_i^2}{n}$ which proves
	\ref{this}
\end{proof}
In the end, we would like to introduce another lemma that will be helpful for the proof. It can be derived from Cauchy Interfacing Theorem and illustrates the relationship between eigenvalues of full sample matrix and eigenvalues of its sub-matrix.
\begin{lemma}\label{interlacing_lemma}
	Assume we have full sample covariance matrix $X$ and its sub-matrix defined in the previous part $X_{SUB}$. There always exists an orthogonal projection $P$ such that:
	$$ P^*\times X\times P = X_{SUB}$$
	If we let $\sigma_{SUB,j}$ represent the $j$-th eigenvalue of $X_{SUB}$ and let $\sigma_{j}$ represent the $j$-th eigenvalue of full sample covariance matrix $X$.
	Also we assume all eigenvalues are sorted in descending order, which means:
	$$ \sigma_1 > \sigma_2 > \cdots > \sigma_n $$
	and
	$$\sigma_{SUB,1} > \sigma_{SUB,2} > \cdots > \sigma_{SUB,n}$$
	Then we have the interlacing property:
	$$ \sigma_{j} > \sigma_{SUB,j} > \sigma_{j+1}, \forall j \in {1,2,\cdots,n-1}$$
\end{lemma}

\begin{proof}
	Without loss of generality, we may assume matrix $X$ is a $n\times n$ matrix and we get $X_{SUB}$ by deleting the last column and last row of matrix $X$.
	
	Now we consider a $n \times n-1$ projection matrix $P$ as following:
	\begin{equation}
	\left[\begin{array}{ccccccccc}
	1 & 0 & \cdots & 0 \\
	0 & 1 & \cdots & 0 \\
	\vdots & \vdots & \cdots & \vdots\\
	0 & 0 & \cdots & 1\\
	0 & 0 & \cdots & 0
	\end{array}\right]
	\end{equation}
	Then we have:
	$$ P^*\times X\times P = X_{SUB}$$
	Next, we apply Cauchy interlacing theorem directly and have:
	$$ \sigma_{j} > \sigma_{SUB,j} > \sigma_{j+1}, \forall j \in {1,2,\cdots,n-1}$$
\end{proof}

\section{Conclusion}
In this paper, we prove an analytical formula to reconstruct spectrum in 2 different cases. Moreover, when we have both $n$ and $p$ go to $\infty$ and their ratio $C$ is a fixed constant, we find the approximation error is of order $o(\frac{1}{C})$. In this way, we show our formula has negligible error when $C$ is large enough.

%
\section*{Conflict of interest}
The authors declare that they have no conflict of interest.

\bibliography{main.bbl}

\end{document}